\documentclass[reqno, 11pt]{amsart}

\usepackage{amsfonts, amsmath, amssymb, amsthm}
\usepackage[colorlinks=true]{hyperref}
\hypersetup{urlcolor=blue, citecolor=red}

\renewcommand{\le}{\leqslant}
\renewcommand{\leq}{\leqslant}
\renewcommand{\ge}{\geqslant}
\renewcommand{\geq}{\geqslant}

\newcommand{\ass}{\operatorname{Ass}}
\newcommand{\ann}{\operatorname{Ann}}

\newcommand{\coht}{\operatorname{coht}}

\newcommand{\supp}{\operatorname{Supp}}

\def\littleo{\operatorname{o}}
\renewcommand{\leq}{\leqslant}
\renewcommand{\geq}{\geqslant}

\newcommand{\nfix}{\mathsf{F}}

\newcommand{\entropy}{\mathsf{h}}
\newcommand{\ugr}{\mathsf{g}}

\newtheorem{theorem}{Theorem}[section]
\newtheorem{lemma}[theorem]{Lemma}
\newtheorem{corollary}[theorem]{Corollary}
\newtheorem{proposition}[theorem]{Proposition}

\theoremstyle{definition}

\newtheorem{example}[theorem]{Example}

\newtheorem{remark}[theorem]{Remark}

\title[The dynamical zeta function]{The dynamical zeta function for commuting automorphisms of zero-dimensional groups}
\author{Richard Miles and Thomas Ward}
\email{richard.miles@math.uu.se}
\email{t.b.ward@durham.ac.uk}


\thanks{The authors gratefully acknowledge the
support of London Mathematical Society
Scheme~IV grant number~$41352$.}

\begin{document}

\begin{abstract}
For a~$\mathbb{Z}^d$-action~$\alpha$ by commuting homeomorphisms of a compact metric space, Lind introduced a dynamical zeta function that generalizes the dynamical zeta function of a single transformation. In this article, we investigate this function when~$\alpha$ is generated by continuous automorphisms of a compact abelian zero-dimensional group. We address Lind's conjecture concerning the existence of a natural boundary for the zeta function and prove this for two significant classes of actions, including both zero entropy and positive entropy examples. The finer structure of the periodic point counting function is also examined and, in the zero entropy case, we
show how this may be severely restricted for subgroups of prime index in~$\mathbb{Z}^d$. We also consider a related open problem concerning the appearance of a natural boundary for the dynamical zeta function of a single automorphism, giving further weight to the P{\'o}lya--Carlson dichotomy proposed by Bell and the authors.
\end{abstract}

\maketitle


\section{Introduction}

For a~$\mathbb{Z}^d$-action~$\alpha$ generated by~$d$ commuting homeomorphisms of a compact metric space~$X$, Lind~\cite{MR1411232} introduced a dynamical zeta function that generalizes the well-known dynamical zeta function of a single transformation introduced by Artin and Mazur~\cite{MR0176482}.
Lind's dynamical zeta function is defined as follows.
Denote the set of finite index subgroups of~$\mathbb{Z}^d$ by~$\mathcal{L}_d$ and for any~$\Lambda\in\mathcal{L}_d$, let~$[\Lambda]=|\mathbb{Z}^d/\Lambda|$. For any~$\mathbf{n}\in\mathbb{Z}^d$, let~$\alpha^{\mathbf{n}}$ denote the element of the action corresponding to~$\mathbf{n}$ and let
$\nfix_\alpha(\Lambda)$ denote the cardinality of the set of points~$x\in X$ with~$\alpha^{\mathbf{n}}(x)=x$ for all~$\mathbf{n}\in \Lambda$. The~\emph{dynamical zeta function} of~$\alpha$ is defined as a formal power series by
\[
\zeta_\alpha(z)
=
\exp\left(
\sum_{\Lambda\in\mathcal{L}_d}
\frac{\nfix_\alpha(\Lambda)}{[\Lambda]}z^{[\Lambda]}
\right).
\]
When~$\alpha$ is generated by a single transformation
(that is, a~$\mathbb{Z}$-action),~$\zeta_\alpha$ agrees with the dynamical zeta function considered by Artin and Mazur.

The number of subgroups
of index~$n$ in~$\mathbb{Z}^d$ is polynomially bounded in~$n$ and so, as observed by Lind~\cite[Th.~5.3]{MR1411232}, it follows from the usual Hadamard formula that~$\zeta_\alpha$ has radius of convergence~$e^{-\ugr(\alpha)}$, where
\[
\ugr(\alpha)=\limsup_{[\Lambda]\rightarrow\infty}
\frac{1}{[\Lambda]}\log\nfix_\alpha(\Lambda)
\]
is the \emph{upper growth rate of periodic points}, which may or may not coincide with the topological entropy~$\entropy(\alpha)$.
Lind conjectures~\cite[Sec.~7]{MR1411232} that for a~$\mathbb{Z}^d$-action with~$d\geqslant 2$, the circle~$|z|=e^{-\mathsf{h}(\alpha)}$ is a natural boundary for~$\zeta_\alpha$ and that~$\zeta_\alpha$ is meromorphic inside this circle. Part of this complexity comes from the subgroup structure of the acting
group: as Lind points out, for~$d\geqslant2$ the zeta function
of the trivial~$\mathbb{Z}^d$-action on a point has a natural boundary.

An important class of~$\mathbb{Z}^d$-actions
particularly amenable to study are the algebraic ones (see the monograph of Schmidt~\cite{MR1345152}
for an overview), that is, those
generated by commuting automorphisms of compact
metric abelian groups. We assume throughout that the group~$X$ under consideration is metrizable. When~$\alpha$ is mixing and~$X$ is connected with finite topological dimension, the first author has shown that Lind's conjecture holds~\cite{MR3336617}. Motivated
in part by examples such as
Ledrappier's~$\mathbb{Z}^2$-action~\cite{MR512106}
that do not fall into this class, our purpose here is to investigate~$\zeta_\alpha$ for algebraic systems where~$X$ is zero-dimensional (equivalently, $X$ is totally disconnected). Notably, Ledrappier's example plays an important role in Lind's initial investigation~\cite[Ex.~3.4]{MR512106} and in many subsequent studies of algebraic ~$\mathbb{Z}^d$-actions. For example, it is a prototypical example of a~$\mathbb{Z}^d$-action of a compact abelian zero-dimensional group where every element of the action has finite entropy, that is, an~\emph{entropy rank one action}~\cite{MR2031042}. If~$\alpha$ is an entropy rank one action, then clearly~$\entropy(\alpha)=0$ if~$d\ge2$.

As well as having the motivation of Lind's conjecture, we also consider a separate problem that concerns the dynamical zeta function for a~$\mathbb{Z}$-action generated by a single automorphism. In this case, if~$\zeta_\alpha$ is not rational, there is strong evidence to suggest that~$\zeta_\alpha$ admits a natural boundary. For automorphisms of connected groups this problem is considered in~\cite{MR3217030}, and the conjectured
dichotomy discussed there is further supported by the following.

\begin{theorem}\label{single_automorphism_theorem}
Let~$X$ be a compact abelian zero-dimensional group and suppose~$\alpha$ is generated by a single ergodic automorphism of~$X$ with finite entropy. If~$\zeta_\alpha$ has radius of convergence~$e^{-\entropy(\alpha)}$, then either
\[
\zeta_\alpha(z)=(1-e^{\entropy(\alpha)} z)^{-1}
\]
or the circle~$|z|=e^{-\entropy(\alpha)}$ is a natural boundary for the function~$\zeta_\alpha$.
\end{theorem}

In~\cite{MR1702897}, the second author considered a natural
uncountable
family of ergodic automorphisms
of zero-dimensional groups parameterised
by a probability space,
and showed that an irrational zeta function should be expected almost surely for an automorphism in this family. Furthermore, the first author shows in~\cite{MR2441142} that any ergodic finite entropy  automorphism of a compact abelian zero-dimensional group has a sequence of periodic point counts that is simply a finite product of the sequences dealt with in~\cite{MR1702897}. It is also shown in~\cite{MR1702897} that~$\zeta_\alpha$ has radius of convergence~$e^{-\entropy(\alpha)}$ almost surely, so from this perspective Theorem~\ref{single_automorphism_theorem} points to the occurrence of a natural boundary for~$\zeta_\alpha$ as typical behaviour, even for a single automorphism.

A crucial step in the proof of the main
result in~\cite{MR3336617} is the
beautiful
theorem of P{\'o}lya and Carlson which states that a power series with integer coefficients and radius of convergence~$1$ is either rational or has the
unit circle as a natural boundary (see~\cite{MR1544479},~\cite{MR1512473} and~\cite{MR2376066}). Using this result, it may be shown that if~$\zeta_\alpha$ has radius of convergence 1 and~$d\geqslant 2$, then~$\zeta_\alpha$ admits the unit circle as a natural boundary (see~Proposition~\ref{pc_route_to_natural_boundary}).
In using the P{\'o}lya--Carlson theorem in these
contexts both the integrality of the coefficients
and the radius of convergence create problems,
and some argument is needed.
Consequently, Lind's conjecture is shown to hold in
the setting of~\cite{MR3336617} by demonstrating that the upper growth rate of periodic points is zero. This is achieved with the help of an upper estimate for~$\nfix_\alpha(\Lambda)$ obtained using techniques from~\cite{MR3074380} and a theorem of Corvaja and Zannier~\cite{MR2130274} concerning bounds on quantities related to greatest common divisors for rings of~$S$-integers. In the setting of this article, similar results are unavailable (see~\cite{MR3082249} and~\cite{MR2052363}). However, by further developing existing periodic point counting techniques for entropy rank one actions~\cite{MR2308145} and by appealing more directly to the geometry of~$\mathcal{L}_d$, we adopt an alternative approach to show that~$\ugr(\alpha)=0$ for all mixing entropy rank one~$\mathbb{Z}^d$-actions by automorphisms of compact abelian zero-dimensional groups, with~$d\geqslant 2$.
Thus, we are able to prove the following.

\begin{theorem}\label{main_result_natural_boundary_entropy_rank_one}
Let~$\alpha$ be a mixing entropy rank one~$\mathbb{Z}^d$-action by automorphisms of a compact abelian zero-dimensional group~$X$, with~$d\geqslant 2$. Then~$\zeta_\alpha$ has radius of convergence~$1$ and the unit circle is a natural boundary for the function.
\end{theorem}

In addition to showing that the upper growth rate of periodic points is zero for entropy rank one actions, we also investigate the finer structure of the periodic point counting function~$\nfix_\alpha:\mathcal{L}_d\rightarrow\mathbb{N}$, with a view to identifying a significant class of subgroups upon which this function is in fact bounded. This emerged as a possibility
from examples such as Ledrappier's. The following result applies to algebraic~$\mathbb{Z}^d$-actions that are Noetherian, which in our
zero-dimensional setting equates to the dynamical property of expansiveness~\cite{MR1036904} (see Section~\ref{periodic_points_section} for a formal definition).

\begin{theorem}\label{main_result_restricted_on_primes}
Let~$\alpha$ be a mixing Noetherian entropy rank one~$\mathbb{Z}^d$-action by automorphisms of a compact abelian zero-dimensional group~$X$, with~$d\geqslant 2$. Then there is a set of indices~$Q$ of positive density in the rational primes such that~$\{\nfix_\alpha(\Lambda):[\Lambda]\in Q\}$ is a finite set. When~$d\geqslant 3$, such a set~$Q$ may be found with density~$1$.
\end{theorem}

The proof of Theorem~\ref{main_result_restricted_on_primes} relies on a strong upper estimate for~$\nfix_\alpha(\Lambda)$ obtained in the proof of  Theorem~\ref{main_result_natural_boundary_entropy_rank_one}
(see Corollary~\ref{minkowski_bound} and Example~\ref{ledrappiers_example}) and contesting lower estimates provided by quantities analogous to greatest common divisors for subrings of global function fields. The proof also involves multiplicative orders of rational primes, and this accounts for the statement involving prime densities (see Remark~\ref{enough_primes_observation}).

Finally, we delve deeper into the dual algebraic structure behind~$\mathbb{Z}^d$-actions in order to obtain a product
decomposition
(Corollary~\ref{pp_count_final_product}) for counting periodic points for~$\mathbb{Z}^2$-actions with positive entropy. This allows us to exploit our earlier estimates for entropy rank one actions in an essential way and to prove Lind's conjecture in the following setting.

\begin{theorem}\label{main_result_natural_boundary_z2}
Let~$\alpha$ be a mixing Noetherian~$\mathbb{Z}^2$-action by automorphisms of a compact abelian zero-dimensional group~$X$. Then~$\zeta_\alpha$ has circle of convergence~$|z|=e^{-\entropy(\alpha)}$, and this is a natural boundary for the function.
\end{theorem}

Notably, in the proof of Theorem~\ref{main_result_natural_boundary_z2},
the P{\'o}lya--Carlson
theorem cannot be used in the same way as in the proof of Theorem~\ref{main_result_natural_boundary_entropy_rank_one}, as~$\zeta_\alpha$ may not have the unit circle as the circle of convergence. Nonetheless, we are able to use a functional equation relating the dynamical zeta function to an ordinary generating function, such that, in combination with our product formula for counting periodic points, the P{\'o}lya--Carlson result can be applied.

The paper is organised as follows. In Section~\ref{periodic_points_section}, we develop the periodic point counting results that lead to Theorem~\ref{main_result_natural_boundary_entropy_rank_one} and Theorem~\ref{main_result_restricted_on_primes}. In Section~\ref{single_automorphisms_section}, using existing periodic point counting formulae, we show how the dichotomy of Theorem~\ref{single_automorphism_theorem} arises via the phenomenon of overconvergence, instead of via the P{\'o}lya--Carlson result. In Section~\ref{examples_section}, we  describe the final steps in the proof of Theorem~\ref{main_result_natural_boundary_entropy_rank_one} and consider some examples. In particular, Example~\ref{shift extension_of_ledrappier} is a mixing~$\mathbb{Z}^3$-action~$\alpha$ on a zero-dimensional group for which~$\entropy(\alpha)=0$ and~$\ugr(\alpha)=\frac{2}{3}\log 2$.
The dynamical zeta function for this example also has a natural boundary, and hence further supports Lind's conjecture. However, it arises in an essentially different way to all the other cases we consider, and cannot be dealt with by similar methods. It is also interesting to note that the value of $\ugr(\alpha)$ for this example is smaller than any possible non-zero entropy of a $\mathbb{Z}^d$--action by automorphisms of a compact abelian zero-dimensional group. Finally,  Section~\ref{positive_entropy_section} is devoted to the proof of Theorem~\ref{main_result_natural_boundary_z2}

\section{Periodic points}\label{periodic_points_section}

In this section we assemble and develop periodic point counting results in order to describe the dynamical zeta function sufficiently
explicitly to demonstrate a natural boundary in a variety of settings.

The commutative algebra used to study a~$\mathbb{Z}^d$-action~$\alpha$ by continuous automorphisms of a compact abelian group~$X$ is now familiar (see, for example, Schmidt's monograph~\cite{MR1345152} and Einsiedler and Lind's paper~\cite{MR2031042} which is particularly useful when each automorphism~$\alpha^\mathbf{n}$ has finite entropy, in which case~$\alpha$ is said to be of \emph{entropy rank one}. The starting point is to notice that
the Pontryagin dual of~$X$, denoted~$M=\widehat{X}$, becomes a module over the Laurent polynomial ring~$R_d=\mathbb{Z}[u_1^{\pm 1},\dots,u_d^{\pm 1}]$ by identifying application of the dual automorphism~$\widehat{\alpha}^\mathbf{n}$ with multiplication by~$u^\mathbf{n}=u_1^{n_1}\cdots u_d^{n_d}$, and extending this in a natural way to polynomials. Conversely, using this multiplication rule, any~$R_d$-module~$M$ defines an action~$\alpha_M$ of~$\mathbb{Z}^d$ by continuous automorphisms of~$X=\widehat{M}$.

Additionally, if~$X$ is assumed to be zero-dimensional, via duality, this means every element of~$M$ has finite
additive order, so every associated
prime ideal~$\mathfrak{p}\in\ass(M)$ contains a rational prime. It will also be useful to refer to the coheight of a prime ideal~$\mathfrak{p}\subset R_d$, denoted~$\coht(\mathfrak{p})$; this coincides with the Krull dimension of the domain~$R_d/\mathfrak{p}$. Einsiedler and Lind show that if~$\alpha$ has entropy rank one and~$X$ is zero-dimensional then~$\coht(\mathfrak{p})\leqslant 1$ for all~$\mathfrak{p}\in\ass(M)$. Unless otherwise stated, we will always assume from now on that~$X$ is zero-dimensional. If the~$R_d$-module~$M$ is Noetherian we also say that~$\alpha=\alpha_M$ is {Noetherian}.
In this case, there is a chain of ideals
\begin{equation}\label{module_chain}
M_0\subset M_1\subset\cdots\subset M_r=M,
\end{equation}
for which~$M_0=\{0\}$ and~$M_i/M_{i-1}\cong R_d/\mathfrak{p}_i$ for a list~$\mathfrak{p}_1,\dots,\mathfrak{p}_r\subset R_d$ of prime ideals which are
either associated primes of~$M$, or contain such a prime. In addition, we may always arrange such a filtration so that that the associated primes appear first and so that those associated primes with the greatest coheight appear foremost amongst the associated primes.

Recall that a \emph{global field}~$\mathbb{K}$ of characteristic~$p>0$ is a finite extension of the rational function field~$\mathbb{F}_p(t)$, where~$t$ is an indeterminate. The \emph{places} of~$\mathbb{K}$ are the equivalence classes of absolute values on~$\mathbb{K}$, which are all non-archimedean.
For example, the \emph{infinite place} of~$\mathbb{F}_p(t)$ is given by~$|f/g|_\infty=p^{\deg(f)-\deg(g)}$ and all other places of~$\mathbb{F}_p(t)$ correspond, in the usual way, to valuation rings obtained by localizing the domain~$\mathbb{F}_p[t]$ at its non-trivial prime ideals (generated by irreducible polynomials). The places of~$\mathbb{K}$ are extensions of those just descibed,
and the set of all such places is denoted~$\mathcal{P}(\mathbb{K})$. Such sets of places provide a foundation for the periodic point counting formulae presented here.

For ease of notation, for any~$\mathbf{n}\in\mathbb{Z}^d$, write~$\nfix_\alpha(\langle \mathbf{n}\rangle)=\nfix_\alpha(\mathbf{n})$. In the case that the action~$\alpha$ is generated by a single ergodic automorphism with finite entropy, a formula for counting periodic points is given in \cite{MR2441142}. In particular, there exist function fields~$\mathbb{K}_1,\dots,\mathbb{K}_r$ of the form~$\mathbb{K}_i=\mathbb{F}_{p(i)}(t)$, where each~$p(i)$ is a rational prime, and sets of finite places~$T_i\subset\mathcal{P}(\mathbb{K}_i)$,~$1\leqslant i\leqslant r$, such that
\begin{equation}\label{periodic_point_formula}
\nfix_\alpha(n)=\prod_{i=1}^{r}\prod_{v\in T_i}|t^n-1|^{-1}_v,
\end{equation}
for all positive integers~$n$. Remark 1 in~\cite{MR2441142} also shows that the infinite places and places corresponding to localization at~$(t)$ do not appear in this product. We assume throughout that absolute values are normalised so that the Artin product formula for global fields holds. Using
the product formula, we may also write~(\ref{periodic_point_formula}) as
\begin{equation}\label{periodic_point_formula_infinite_places}
\nfix_\alpha(n)=\prod_{i=1}^{r}\prod_{v\in S_i}|t^n-1|_v,
\end{equation}
where~$S_i=\mathcal{P}(\mathbb{K}_i)\setminus T_i$.

Suppose~$\alpha$ is an entropy rank one~$\mathbb{Z}^d$-action and that the module~$M$ is Noetherian. If there exists~$\mathbf{n}\in\mathbb{Z}^d$ for which~$\alpha^\mathbf{n}$ is ergodic then~\cite[Sec.~2]{MR2308145} shows that~$\coht(\mathfrak{p})=1$ for all~$\mathfrak{p}\in\ass(M)$. In this case, the field of fractions of~$R_d/\mathfrak{p}$ is again a global field of positive characteristic which we denote by~$\mathbb{K}(\mathfrak{p})$. Furthermore, for any~$\mathbf{n}\in\mathbb{Z}^d$ for which~$\alpha^\mathbf{n}$ is ergodic,~\cite[Sec.~3]{MR2308145} shows that
\begin{equation}\label{what_youre_doing_in_my_head}
\nfix_\alpha(\mathbf{n})
=
\prod_{\mathfrak{p}\in\ass(M)}
\prod_{v\in S(\mathfrak{p})}
|\overline{u}^\mathbf{n}-1|^{m(\mathfrak{p})}_v
\end{equation}
where~$\overline{u}^\mathbf{n}$ denotes the image of~$u_1^{n_1}\cdots u_d^{n_d}$ in the domain~$R_d/\mathfrak{p}$,~$m(\mathfrak{p})$ is the dimension of the~$\mathbb{K}(\mathfrak{p})$-vector space~$M\otimes\mathbb{K}(\mathfrak{p})$ and
\begin{equation}\label{places_definition}
S(\mathfrak{p})=\{v\in\mathcal{P}(\mathbb{K}(\mathfrak{p})):|R_d/\mathfrak{p}|_v\mbox{ is an unbounded subset of }\mathbb{R}\}.
\end{equation}
Note that the set~$S(\mathfrak{p})$ is a finite because~$R_d/\mathfrak{p}$ is finitely generated.

In a more general setting, it is difficult to obtain such
an explicit formula, although coarser estimates are also sometimes useful, such as the following.

\begin{lemma}\label{i_dont_mind_if_i_lose}
Let~$\alpha$ be an
entropy rank one~$\mathbb{Z}^d$-action by automorphisms of a compact abelian zero-dimensional group~$X$ and let~$M$ denote the dual~$R_d$-module. Then there is a~$\mathbb{Z}^d$-action~$\beta$ by automorphisms of a compact abelian zero-dimensional group corresponding via duality to a Noetherian submodule of $M$ such that, for each~$\mathbf{n}\in\mathbb{Z}^d$ for which~$\alpha^\mathbf{n}$ is ergodic,~$\beta^\mathbf{n}$ is ergodic and
\[
\nfix_\alpha(\mathbf{n})\leqslant\nfix_\beta(\mathbf{n})<\infty.
\]
\end{lemma}

\begin{proof}
If there is no~$\mathbf{n}\in\mathbb{Z}^d$ for which~$\alpha_M^\mathbf{n}$ is ergodic, or if~$M$ is Noetherian, then there is nothing to prove. Hence, assume~$\alpha_M^\mathbf{n}$ is ergodic and~$M$ is not Noetherian. Since~$M$ corresponds to an entropy rank one action, so does every submodule~$L\subset M$ and quotient~$M/L$. Also, since~$M$ can be expressed as a
countable increasing
union of Noetherian submodules, \cite[Prop.~4.4 and Prop.~6.1]{MR2031042} show that every prime ideal~$\mathfrak{p}\in\ass(M)$ has~$\coht(\mathfrak{p})\leqslant 1$. If~$\coht(\mathfrak{p})=0$, then~$\mathfrak{p}$ is maximal,~$R_d/\mathfrak{p}$ is a finite field, and subsequently~$\alpha_{R_d/\mathfrak{p}}^\mathbf{n}$ is not ergodic.
Since there is a submodule of~$M$ isomorphic to~$R_d/\mathfrak{p}$, this contradicts ergodicity of~$\alpha_M^\mathbf{n}$, by~\cite[Prop.~6.6]{MR1345152}. Therefore,
~$\coht(\mathfrak{p})=1$ for all~$\mathfrak{p}\in\ass(M)$.

If~$L\subset M$ is a Noetherian submodule,~$\mathfrak{p}\in\ass(M/L)$ and~$\coht(\mathfrak{p})=1$, then for some~$x\in M$,~$\ann(x+L)=\mathfrak{p}$. So,~$\mathfrak{p}$ contains~$\ann(x)$ and~$\mathfrak{p}\in\supp(Rx)$. Since all associated primes of~$Rx$ have coheight 1,~$\ass(Rx)$ coincides with the elements of~$\supp(Rx)$ with coheight 1, by~\cite[Th.~6.5(iii)]{MR1011461}. Hence,
\begin{equation}\label{from_a_man_whose_plan_is_already_wrecked}
\mathfrak{p}\in\ass(M/L)\mbox{ and }\coht(\mathfrak{p})=1\Rightarrow\mathfrak{p}\in\ass(M)
\end{equation}
We claim that there is a Noetherian submodule~$L\subset M$ such that~$\ass(M/L)$ is comprised entirely of maximal ideals. To see this, choose a submodule~$L\subset M$ such that~$L\cong R_d/\mathfrak{p}$ for some~$\mathfrak{p}\in \ass(M)$. Since~$\alpha_{R_d/\mathfrak{p}}^\mathbf{n}$ is an ergodic automorphism of a zero-dimensional group,~$\mathsf{h}(\alpha_{R_d/\mathfrak{p}}^\mathbf{n})\geqslant\log 2$. Now, if there exists~$\mathfrak{p}'\in\ass(M/L)$ with~$\coht(\mathfrak{p}')=1$, then~(\ref{from_a_man_whose_plan_is_already_wrecked}) shows that~$\mathfrak{p}'\in\ass(M)$, and there is a Noetherian  submodule~$L'\subset M$ containing~$L$ with~$L'/L\cong R/\mathfrak{p}'$. Furthermore, the entropy addition formula shows that~$\mathsf{h}(\alpha_{L'}^\mathbf{n})\geqslant 2\log 2$. If we can find no such prime~$\mathfrak{p}'$, then~$\ass(M/L)$ is comprised entirely of maximal ideals. If we repeat the process inductively with~$L'$ in place of~$L$, it must eventually terminate with a Noetherian submodule~$L\subset M$ with the claimed property, otherwise we contradict the assumption that~$\mathsf{h}(\alpha_M^\mathbf{n})<\infty$. Note also that~$L$ is independent of the specific choice of ergodic automorphism~$\alpha_M^\mathbf{n}$.

We may now find a chain of modules
\[
\{L\}=M_0\subset M_1\subset M_2\subset\dots
\]
with~$M=\bigcup_{i\ge0}M_i$ and~$M_i/M_{i-1}$ finite for all~$i\geqslant 1$. Following the approach of~\cite[Sec.~3]{MR2308145}, we have~$\nfix_\alpha(\mathbf{n})=|M/(u^\mathbf{n}-1)M|$,
\begin{equation}\label{well_i_burned}
M/(u^\mathbf{n}-1)M\cong\varinjlim M_i/((u^\mathbf{n}-1)M\cap M_i),
\end{equation}
and
\begin{equation}\label{all_my_travellers_cheques}
|M_i/((u^\mathbf{n}-1)M\cap M_i)|\leqslant|M_i/(u^\mathbf{n}-1)M_i)|.
\end{equation}
Since each~$M_i$ is a Noetherian~$R_d$-module corresponding to an entropy rank one action, and since each factor~$M_i/M_{i-1}$ is finite, the proof of~\cite[Th.~3.2]{MR2308145} shows that
\[
|M_i/(u^\mathbf{n}-1)M_i)|=|L/(u^\mathbf{n}-1)L|
\]
for all~$i\geqslant 0$. It follows from~(\ref{well_i_burned}) and~(\ref{all_my_travellers_cheques}) that~$|M/(u^\mathbf{n}-1)M|$ is bounded above by~$|L/(u^\mathbf{n}-1)L)|$. If we let~$\beta=\alpha_L$ be the the action corresponding to the Noetherian~$R_d$-module~$L$, then we have~$\nfix_\beta(\mathbf{n})=|L/(u^\mathbf{n}-1)L|$, and the required result follows.
\end{proof}

This leads to the following uniformity result.

\begin{theorem}\label{cos_if_i_win_ill_be_so_confused}
If~$\alpha$ is an entropy rank one~$\mathbb{Z}^d$-action by continuous automorphisms of a compact abelian zero-dimensional group~$X$, then there is a constant~$C>0$ such that, for all~$\mathbf{n}\in\mathbb{Z}^d$ for which~$\alpha^\mathbf{n}$ is ergodic,
\[
\nfix_\alpha(\mathbf{n})\leqslant C^{||\mathbf{n}||}.
\]
\end{theorem}

\begin{proof}
Since we simply require an upper estimate for~$\nfix_\alpha(\mathbf{n})$, using Lemma~\ref{i_dont_mind_if_i_lose}, we may as well assume that~$(X, \alpha)$ corresponds to a Noetherian module~$M$. In this case, we may apply~(\ref{what_youre_doing_in_my_head}) to calculate~$\nfix_\alpha(\mathbf{n})$.

Considering any place~$v$ appearing in~(\ref{what_youre_doing_in_my_head}), the ultrametric inequality shows that
\[
|\overline{u}^\mathbf{n}-1|_v\leqslant
\prod_{i=1}^d\max\{|\overline{u}_i|_v,|\overline{u}_i|_v^{-1}\}^{|n_i|}.
\]
Let
\[
\lambda=\max\{\max\{|\overline{u}_i|_v,|\overline{u}_i|_v^{-1}\}:i=1\dots d, v\in S(\mathfrak{p}),  \mathfrak{p}\in\ass(M)\}
\]
and~$E=\sum_{\mathfrak{p}\in\ass(M)}{m(\mathfrak{p})}|S(\mathfrak{p})|$. Then the right-hand side of~(\ref{what_youre_doing_in_my_head}) is at most
\[
\lambda^{E\cdot||\mathbf{n}||_1},
\]
where~$||\mathbf{n}||_1=\sum_{i=1}^d|n_i|$. Since~$||\mathbf{n}||_1\leqslant\sqrt{d}||\mathbf{n}||$, the required result follows from~(\ref{what_youre_doing_in_my_head}) with~$C=\lambda^{E\sqrt{d}}$.
\end{proof}

We now turn our attention to estimating the size of~$\nfix_\alpha(\Lambda)$,~$\Lambda\in\mathcal{L}_d$. For any~$\Lambda\in\mathcal{L}_d$, let
\begin{equation}\label{pp_ideal_definition}
\mathfrak{b}_\Lambda=(\overline{u}^\mathbf{n}-1:\mathbf{n}\in \Lambda)\subset R_d.
\end{equation}
Note that this ideal is identical to~$(\overline{u}^\mathbf{n}-1:\mathbf{n}\in \Lambda')$, where~$\Lambda'$ is any set of generators for the group~$\Lambda$.

To begin with, we have the following consequence of Theorem~\ref{cos_if_i_win_ill_be_so_confused}.

\begin{corollary}\label{minkowski_bound}
Let~$\alpha$ be an entropy rank one~$\mathbb{Z}^d$-action by continuous automorphisms of a compact abelian zero-dimensional group~$X$. Then there is a constant~$C>0$ such that, for all~$\Lambda\in\mathcal{L}_d$ with~$\alpha^\mathbf{n}$ ergodic for all ~$\mathbf{n}\in\Lambda\setminus\{\mathbf{0}\}$,
\[
\nfix_{\alpha}(\Lambda)\leqslant C^{\sqrt{d}[\Lambda]^{1/d}},
\]
\end{corollary}

\begin{proof}
By a standard application of Minkowski's Theorem, for any~$\Lambda\in\mathcal{L}_d$, we can find a non-trivial~$\mathbf{m}\in \Lambda$ such that
\[
||\mathbf{m}||\leqslant\sqrt{d}[\Lambda]^{1/d}.
\]
If~$\Lambda\in\mathcal{L}_d$ is such that~$\alpha^\mathbf{n}$ ergodic for all~$\mathbf{n}\in\Lambda\setminus\{\mathbf{0}\}$, then since
\[
\nfix_{\alpha}(\Lambda)\leqslant\nfix_{\alpha}(\mathbf{m}),
\]
the required result follows from Theorem~\ref{cos_if_i_win_ill_be_so_confused}.
\end{proof}

The following is a well-known consequence of Pontryagin duality.

\begin{lemma}
Let~$\alpha$ be a~$\mathbb{Z}^d$-action by continuous automorphisms of a compact abelian group~$X$ and let~$M$ be the corresponding dual module.
Then, whenever either~$\nfix_\alpha(\Lambda)$ or~$|M/\mathfrak{b}_\Lambda M|$ is finite, both these quantities coincide.
\end{lemma}

\begin{proof}
See~\cite[Lem.~7.2]{MR1062797}.
\end{proof}

Using the isomorphism theorems (see~\cite[Lem.~7.5]{MR1062797}), we also have the following basic factorization result.

\begin{lemma}\label{lsw_basic_pp_estimate}
Let~$L\subset M$ be~$R_d$-modules and~$\mathfrak{b}\subset R_d$ an ideal such that~$M/\mathfrak{b}M$ is finite, then
\[
|M/\mathfrak{b}M|
= \left|\frac{M/L}{\mathfrak{b}(M/L)}\right|\frac{|L/\mathfrak{b}L|}{|(\mathfrak{b}M\cap L)/\mathfrak{b}L|}.
\]
\end{lemma}

When~$M$ is a module of the form~$D=R_d/\mathfrak{p}$, with~$\mathfrak{p}$ a prime ideal of coheight 1,~$D$ is a finitely generated domain of Krull dimension one over~$\mathbb{Z}$ or~$\mathbb{F}_p$ for some rational prime~$p$. Every  proper quotient of~$D$ by an ideal is finite, the integral closure~$E$ of~$D$ is a finitely generated~$D$-module, and the quotient~$E/D$ is finite. This leads to the following useful result.

\begin{lemma}\label{integral_closure_trick}
Let~$\mathfrak{p}\subset R_d$ be a prime ideal with~$\coht(\mathfrak{p})=1$ and let~$D=R_d/\mathfrak{p}$. Then there exist constants~$A,B>0$ such that
for any proper ideal~$\mathfrak{a}\subset D$
\[
A|E/\mathfrak{a}E|\leqslant|D/\mathfrak{a}|\leqslant B|E/\mathfrak{a}E|
\]
where~$E$ is the integral closure of~$D$ in its
field of fractions.
\end{lemma}

\begin{proof}
Firstly, since~$D$ arises as a quotient ring of~$R_d$, both~$D$ and~$E$ may be considered as~$R_d$-modules in a natural way and Lemma~\ref{lsw_basic_pp_estimate} shows  that
\[
|E/\mathfrak{a}E|\leqslant|(E/D)/\mathfrak{a}(E/D)||D/\mathfrak{a}|\leqslant |E/D||D/\mathfrak{a}|.
\]
Furthermore, since~$E$ is generated by finitely many elements in the
field of fractions of~$D$, there exists~$x\in D$ which will clear denominators so that~$xE$ is a finite index submodule of~$D$. Therefore, in a similar way, we also have
\[
|D/\mathfrak{a}|\leqslant|D/xE||xE/\mathfrak{a}xE|=|D/xE||E/\mathfrak{a}E|,
\]
where the final equality holds as~$xE\cong E$. Hence,
\[
A|E/\mathfrak{a}E|\leqslant|D/\mathfrak{a}|\leqslant B|E/\mathfrak{a}E|
\]
where~$A=|E/D|^{-1}$ and~$B=|D/xE|$ are independent of the ideal~$\mathfrak{a}$.
\end{proof}

\begin{lemma}\label{I_caught_sleep}
Let~$\mathfrak{p}\subset R_d$ be a prime ideal with~$\coht(\mathfrak{p})=1$ and let~$D=R_d/\mathfrak{p}$. Then there exist constants~$A,B>0$ such that
for any~$\Lambda\in\mathcal{L}_d$ and any set of generators~$\Lambda'\subset\Lambda$ with~$\overline{u}^\mathbf{n}\neq 1$ for all~$\mathbf{n}\in\Lambda'$,
\[
A\prod_{v\in T(\mathfrak{p})}
\min_{\mathbf{n}\in\Lambda'}\{|\overline{u}^\mathbf{n}- 1|^{-1}_v\}
\leqslant
|D/\mathfrak{b}_\Lambda D|
\leqslant
B\prod_{v\in T(\mathfrak{p})}
\min_{\mathbf{n}\in\Lambda'}\{|\overline{u}^\mathbf{n}- 1|^{-1}_v\},
\]
where~$T(\mathfrak{p})=\mathcal{P}(\mathbb{K}(\mathfrak{p}))\setminus S(\mathfrak{p})$. Furthermore, if~$D$ is integrally closed, then
$|D/\mathfrak{b}_\Lambda D|=\prod_{v\in T(\mathfrak{p})}
\min_{\mathbf{n}\in\Lambda'}\{|\overline{u}^\mathbf{n}- 1|^{-1}_v\}$.
\end{lemma}

\begin{proof}
If~$D$ is integrally closed in its
field of fractions~$\mathbb{K}(\mathfrak{p})$, the required result is obtained using a completely analogous method to~\cite[Th.~4.1]{MR3074380} which, with a straightforward adaption to our setting, shows that
\begin{equation}\label{to_the_falls_that_they_can_lead}
|D/\mathfrak{b}_\Lambda D|=\prod_{v\in T(\mathfrak{p})}
\min_{\mathbf{n}\in\Lambda'}\{|\overline{u}^\mathbf{n}- 1|^{-1}_v\}.
\end{equation}
The main idea in the proof of this result is to identify factors and estimate multiplicities in a composition series for~$D/\mathfrak{b}_\Lambda D$ using places of the field~$\mathbb{K}(\mathfrak{p})$. If~$D$ is not integrally closed, then Lemma~\ref{integral_closure_trick} shows that there exist constants~$A,B>0$, independent of~$\mathfrak{b}_\Lambda$, such that
\[
A|E/\mathfrak{b}_\Lambda E|\leqslant|D/\mathfrak{b}_\Lambda D|\leqslant B|E/\mathfrak{b}_\Lambda E|
\]
where~$E$ is the integral closure of~$D$ in~$\mathbb{K}(\mathfrak{p})$. Since
$|E/\mathfrak{b}_\Lambda E|$ is given by the product in~(\ref{to_the_falls_that_they_can_lead}),
the required result follows.
\end{proof}

As in~\cite{MR1411232}, using Hermite normal form, we recall that
any~$\Lambda\in\mathcal{L}_d$ has a unique representation as the image of~$\mathbb{Z}^d$ under a matrix of the form
\begin{equation}\label{hermite_matrix}
\left(
\begin{matrix}
a_1 & b_{12} & b_{13} & \dots & b_{1d}\\
0 & a_2 & b_{23} & \dots & b_{2d}\\
0 & 0 & a_3 & \dots & b_{3d}\\
\vdots & \vdots & \vdots & \ddots & \vdots\\
0 & 0 & 0 & \dots & a_d
\end{matrix}
\right)
\end{equation}
where~$a_m\geqslant 1$ for~$1\leqslant m\leqslant d$,~$0\leqslant b_{mn}\leqslant a_m-1$ for~$m+1\leqslant n \leqslant d$ and~$a_1 a_2\cdots a_d =[\Lambda]$. We now focus on identifying non-trivial collections of subgroups~$\Lambda\subset\mathcal{L}_d$ for which~$\nfix_\alpha(\Lambda)$ is severely restricted. For any rational primes~$p, q$ with~$q\neq p$, let~$m_p(q)$ denote the multiplicative order of~$p$ modulo~$q$.
Let~$\mathbb{P}$ be the set of rational primes and
for any finite set~$P\subset\mathbb{P}$ and any~$\varepsilon>0$, define
\begin{equation*}
\mathcal{L}_d(P, \varepsilon)=\{\Lambda\in\mathcal{L}_d:[\Lambda]\in \mathbb{P}\setminus P\mbox{ and}\,m_p([\Lambda])> [\Lambda]^{\frac{1}{d}+\varepsilon}\mbox{ for all }p\in P\}
\end{equation*}
and let
\begin{equation}\label{rational_primes_for_the_action}
\mathbb{P}(\alpha)=
\{
p\in\mathbb{P}: p\in\mathfrak{p}\mbox{ for some }\mathfrak{p}\in\ass(M)
\}.
\end{equation}
\begin{remark}\label{enough_primes_observation}
It is a non-trivial problem to determine which subgroups of prime index comprise~$\mathcal{L}_d(\mathbb{P}(\alpha), \varepsilon)$. Nonetheless, if we take, for example,~$\varepsilon=0.1$, the set of such prime indices will have positive density in~$\mathbb{P}$. The authors would like to thank P\"ar Kurlberg for pointing out the following explanation. The background and relevant number theoretic results used are summarised in the introduction of~\cite{MR2167719}. To begin with, consider the case~$d=2$. There is a sequence of primes~$q$ of positive density in~$\mathbb{P}$ such that~$q-1$ has a prime divisor~$r$ larger than~$q^{0.6}$. So, for a given~$p\in\mathbb{P}(\alpha)$,~$m_p(q)$ is possibly divisible by~$r$. If not, then~$m_p(q)$ is slightly smaller than~$\sqrt{q}$, smaller than say~$q^{0.499}$.
However, given~$x>0$, the number of primes~$q<x$ such that~$q$ divides~$m_p(q)$ when~$m_p(q)<q^{0.499}$ is~$\littleo(x/\log(x))$, that is, the set of such primes has zero density in~$\mathbb{P}$. Since~$\mathbb{P}(\alpha)$ is finite, the set of such unwanted primes~$q$ is still of zero density if we consider all~$p\in\mathbb{P}(\alpha)$ simultaneously, and we retain a set of primes~$q$ of positive density such that~$m_p(q)>q^{0.6}$ for all~$p\in\mathbb{P}(\alpha)$. When~$d\geqslant 3$ and~$\varepsilon$ is small, it follows more straightforwardly from such arguments that~$\mathcal{L}_d(\mathbb{P}(\alpha), \varepsilon)$ is determined by a set of prime indices of density one in~$\mathbb{P}$, as we are then interested in the less delicate case of primes~$q$ that satisfy,~$m_p(q)>q^{\frac{1}{3}+\varepsilon}$, for example.
\end{remark}

The observations above together with the following give~Theorem~\ref{main_result_restricted_on_primes}.

\begin{theorem}
Let~$\alpha$ be a mixing Noetherian entropy rank one~$\mathbb{Z}^d$-action by continuous automorphisms of a compact abelian zero-dimensional group~$X$, with~$d\geqslant 2$.
Then for any~$\varepsilon>0$,
$\{\nfix_\alpha(\Lambda):\Lambda\in\mathcal{L}_d(\mathbb{P}(\alpha), \varepsilon)\}$
is a finite set.
\end{theorem}

\begin{proof}
Firstly, Theorem~\ref{minkowski_bound} shows that there is a constant~$C_0>0$ such that
\begin{equation}\label{developed_minkowski_bound}
\nfix_{\alpha}(\Lambda)
\leqslant C_0^{\sqrt{d}[\Lambda]^{1/d}}.
\end{equation}
Since the action is both mixing and of entropy rank one,~$\coht(\mathfrak{p})=1$ for all~$\mathfrak{p}\in\ass(M)$. Any remaining prime ideal~$\mathfrak{q}\subset R_d$ appearing in a filtration of~$M$ is maximal, and~$R_d/\mathfrak{q}$ is a finite field.
By taking a prime filtration of~$M$ as in~(\ref{module_chain}), and using Lemma~\ref{lsw_basic_pp_estimate} inductively, we also find a constant~$C_1>0$ such that for all~$\Lambda\in\mathcal{L}_d$,
\begin{equation}\label{upper_product_estimate}
\nfix_\alpha(\Lambda) \leqslant
C_1 \prod_{\mathfrak{p}\in\ass(M)}
\left|\frac{R_d/\mathfrak{p}}{\mathfrak{b}_\Lambda(R_d/\mathfrak{p})}\right|^{m(\mathfrak{p})}
\end{equation}
where~$m(\mathfrak{p})$ is the dimension of the~$\mathbb{K}(\mathfrak{p})$-vector space~$M\otimes\mathbb{K}(\mathfrak{p})$.

The mixing assumption means that in each domain~$R_d/\mathfrak{p}$,~$\overline{u}^\mathbf{n}\neq 1$ for all non-zero~$\mathbf{n}\in\mathbb{Z}^d$ (see \cite[Sec.~6]{MR1345152}), so Lemma~\ref{I_caught_sleep} applies and for each~$\mathfrak{p}\in\ass(M)$, there exist constants~$A(\mathfrak{p}), B(\mathfrak{p})>0$, such that
\begin{equation}\label{upper_and_lower_estimates}
A(\mathfrak{p})
\prod_{v\in T(\mathfrak{p})}
\min_{\mathbf{n}\in\Lambda'}\{|\overline{u}^\mathbf{n}- 1|^{-1}_v\}
\leqslant
\left|\frac{R_d/\mathfrak{p}}{\mathfrak{b}_\Lambda(R_d/\mathfrak{p})}\right|
\leqslant
B(\mathfrak{p})
\prod_{v\in T(\mathfrak{p})}
\min_{\mathbf{n}\in\Lambda'}\{|\overline{u}^\mathbf{n}- 1|^{-1}_v\}
\end{equation}
where~$T(\mathfrak{p})=\mathcal{P}(\mathbb{K}(\mathfrak{p}))\setminus S(\mathfrak{p})$, and~$\Lambda'$ is any set of generators for~$\Lambda$. Note that each set of places~$S(\mathfrak{p})$ defined by~(\ref{places_definition}) is finite.

Let~$\mathfrak{p}\in\ass(M)$ and let~$p\in \mathbb{P}(\alpha)$ be the unique rational prime in~$\mathfrak{p}$. Suppose~$\Lambda\in\mathcal{L}_d(\alpha, \varepsilon)$ is given by the matrix~(\ref{hermite_matrix}) and~$[\Lambda]=q$, for a prime~$q$ large enough so that
\begin{equation}\label{in_the_gods}
q^{\frac{1}{d}+\varepsilon}\log p > q^{1/d}\sqrt{d}\log C_0 - \log A
\end{equation}
where~$A=\min\{A(\mathfrak{p}): \mathfrak{p}\in\ass(M)\}$.

Denote the image of~$u_i$ in~$R_d/\mathfrak{p}$ by~$\overline{u}_i$,~$1\leqslant i\leqslant d$. If~$\overline{u}_1$ is algebraic over~$\mathbb{F}_p$, then~$\mathbb{F}_p(\overline{u}_1)$ is a finite field and~$\overline{u}_1^{m}=1$ for some~$m\geqslant 1$, contradicting the assumption that~$\overline{u}^\mathbf{n}\neq 1$ for all non-zero~$\mathbf{n}\in\mathbb{Z}^d$. Therefore, we may identify~$\overline{u}_1$ with an indetermine~$t$ so that~$\mathbb{F}_p[\overline{u}_1]\cong\mathbb{F}_p[t]$. Let~$T_1(\mathfrak{p})$ denote the set of places in~$T(\mathfrak{p})$ extending the place of~$\mathbb{F}_p(t)$ corresponding to the polynomial~$t-1$, and let~$T_2(\mathfrak{p})=T(\mathfrak{p})\setminus T_1(\mathfrak{p})$.

By the definition of the matrix~(\ref{hermite_matrix}),~$a_j=q$ for precisely one index~$1\leq j\leqslant d$, and~$a_i=1$ for all~$i\neq j$. If~$j\neq 1$, then~$(1,0,\dots,0)\in\Lambda$, and so by~(\ref{upper_and_lower_estimates}), the right-hand side of~(\ref{upper_product_estimate}) is at most
\[
C_2
=
C_1
\prod_{\mathfrak{p}\in\ass(M)}
\left(
B(\mathfrak{p})\prod_{v\in T_1(\mathfrak{p})}|t-1|_v^{-1}
\right)^{m(\mathfrak{p})}
\]
as~$|t-1|_v=1$ for all~$v\in T_2(\mathfrak{p})$,~$\mathfrak{p}\in\ass(M)$.

Suppose now that~$j=1$, so~$(q,0,\dots,0)\in\Lambda'$, where~$\Lambda'$ is the set of generators given by multiplying the standard basis vectors in~$\mathbb{Z}^d$ by~(\ref{hermite_matrix}). In the product in~(\ref{upper_and_lower_estimates}), note
that each factor is a positive integer as~$|\overline{u}^\mathbf{n}|_v=1$ for all~$v\in T(\mathfrak{p})$, by~(\ref{places_definition}). Therefore,~(\ref{upper_and_lower_estimates}) also gives,
\begin{equation}\label{and_bravely_in_her_bosom_there}
\left|\frac{R_d/\mathfrak{q}}{\mathfrak{b}_\Lambda(R_d/\mathfrak{q})}\right|
\geqslant
A\prod_{v\in T_2(\mathfrak{q})}
\min_{\mathbf{n}\in\Lambda'}\{|\overline{u}^\mathbf{n}- 1|^{-1}_v\}.
\end{equation}
Suppose, for a contradiction, that the product on the right-hand side of~(\ref{and_bravely_in_her_bosom_there}) is not equal to 1. Then, there exists~$v\in T_2(\mathfrak{q})$ such that~$|\overline{u}^\mathbf{n}- 1|^{-1}_v>1$ for all~$\mathbf{n}\in\Lambda'$. In particular,
\[
|\overline{u}^{(q,0,\dots,0)}- 1|_v=|t^q-1|_v > 1.
\]
Let~$\mathfrak{K}_v$ denote the residue field at~$v$ and consider the prime factorization of~$t^q-1$ in~$\mathbb{F}_p[t]$, which comprises the factor~$t-1$ and the irreducible factors of the~$q$-th cyclotomic polynomial, each of which has degree~$m_p(q)$, by \cite[Th.~2.47]{MR1294139}. Since~$v\in T_2(\mathfrak{p})$ and~$|t^q-1|_v > 1$,~$v$ must extend a place of~$\mathbb{F}_p[t]$ corresponding to one of these these latter factors,~$f$ say. Furthermore, since~$\mathfrak{K}_v$ is an extension of the residue field~$\mathfrak{K}_{(f)}$,
\[
|\mathfrak{K}_v|\geqslant|\mathfrak{K}_{(f)}|=p^{m_p(q)}.
\]
Therefore,  for all~$x\in\mathbb{K}$,
\[
|x|_v > 1\Rightarrow |x|_v\geqslant p^{m_p(q)}.
\]
Hence, if~$|\overline{u}^\mathbf{n}- 1|_v^{-1}>1$ for all~$\mathbf{n}\in\Lambda'$, then
\[
\min_{\mathbf{n}\in\Lambda'}\{|\overline{u}^\mathbf{n}- 1|^{-1}_v\}\geqslant p^{m_p(q)}.
\]
Furthermore, by the definition of~$\mathcal{L}_d(\alpha, \varepsilon)$,~$m_p(q)>q^{\frac{1}{d}+\varepsilon}$, so
the inequalities~(\ref{in_the_gods})
and~(\ref{and_bravely_in_her_bosom_there}) give
\[
\left|\frac{R_d/\mathfrak{q}}{\mathfrak{b}_\Lambda(R_d/\mathfrak{q})}\right|
\geqslant
Ap^{q^{\frac{1}{d}+\varepsilon}}
>
C^{\sqrt{d}q^{1/d}},
\]
which contradicts~(\ref{developed_minkowski_bound}). Thus, for all~$v\in T_2(\mathfrak{p})$,~$\min_{\mathbf{n}\in\Lambda'}
\{|\overline{u}^\mathbf{n}- 1|^{-1}_v\}=1$
and~(\ref{upper_and_lower_estimates}) now gives
\begin{eqnarray*}
\left|\frac{R_d/\mathfrak{q}}{\mathfrak{b}_\Lambda(R_d/\mathfrak{q})}\right|
&\leqslant&
B(\mathfrak{p})\prod_{v\in T_1(\mathfrak{p})}
\min_{\mathbf{n}\in\Lambda'}\{|\overline{u}^\mathbf{n}- 1|^{-1}_v\}\\
&\leqslant&
B(\mathfrak{p})\prod_{v\in T_1(\mathfrak{p})}|t^q- 1|^{-1}_v\\
&=&
B(\mathfrak{p})\prod_{v\in T_1(\mathfrak{p})}|t- 1|^{-1}_v.
\end{eqnarray*}
It follows that~$\nfix_\alpha(\Lambda)\leqslant C_2$ by~(\ref{upper_product_estimate}).
Since there are only finitely many primes~$q$ that do not satisfy~(\ref{in_the_gods}), and only finitely many subgroups of~$\mathbb{Z}^d$ of any given index, the required result now follows.
\end{proof}

\section{The zeta function for a single automorphism}\label{single_automorphisms_section}

Let~$X$ be a compact abelian zero-dimensional group and suppose the action~$\alpha$ is generated by a single ergodic automorphism of~$X$ with finite entropy. A full shift on~$p$ symbols is the simplest example of such a system, and this has zeta function~$\zeta_\alpha(z)=(1-pz)^{-1}$. In this section, we consider typical extensions of full shifts. We begin with the following example, which is the invertible extension of~\cite[Ex.~8.5]{MR1461206} and is naturally related to Ledrappier's example (Example~\ref{ledrappiers_example}). It may be considered to be one of the most fundamental examples given by a single automorphism of a compact abelian zero-dimensional group that is not a full shift.

\begin{example}[An extension of a full~$p$-shift]\label{typical_full_shift_extension}
Let~$p$ be a rational prime and consider the compact abelian group
\[
X=\{x=(x_{i,j})\in\mathbb{F}_p^{\mathbb{Z}^2}:
x_{i+1,j}-x_{i,j+1}-x_{i,j}=0\mbox{ for all }(i,j)\in\mathbb{Z}^2\}.
\]
When~$p=2$, this group is the shift space underlying Ledrappier's example. In that case, a~$\mathbb{Z}^2$-action is generated by the horizontal and vertical shifts, defined by~$(x)_{i,j}\mapsto x_{i+1,j}$ and~$(x)_{i,j}\mapsto x_{i,j+1}$ respectively. Dually,~$\widehat{X}$ can be identified with the ring~$D=\mathbb{F}_p[t^{\pm 1},(t-1)^{\pm 1}]$ and the map dual to the horizontal shift identifies with multiplication by~$t$ on~$D$. Denote the~$\mathbb{Z}$-action generated by horizontal shift by~$\alpha$. Then~(\ref{periodic_point_formula_infinite_places}) gives
\[
\nfix_\alpha(n)=|t^n-1|_\infty|t^n-1|_{(t)}|t^n-1|_{(t-1)}=p^{n-\nu(n)},
\]
where~$\nu(n)=|n|_p^{-1}$. Therefore,~$\zeta_\alpha$ has radius of convergence~$p^{-1}$. Let~$a_n=p^{-\nu(n)}$ and~$F(z)=\sum_{n\geqslant 1}a_nz^n$. Clearly,~$\limsup_{n\rightarrow\infty}a_n^{1/n}=1$ and consequently~$F$ has radius of convergence 1. Furthermore,~$F(pz)=z\zeta_\alpha'(z)/\zeta_\alpha(z)$,
so if~$\zeta_\alpha$ has analytic continutation beyond the circle~$|z|=p^{-1}$, then~$F$ has analytic continutation beyond the unit circle. However, for~$n_k=p^k$, we have
\[
\limsup_{n_k\rightarrow\infty}a_{n_k}^{1/n_k}=p^{-1},
\]
and so~$F$ contains a sequence of partial sums that is uniformly convergent for~$|z|<p$. Hence, the series~$F$ is overconvergent,
so the series may be written as
a sum of a series convergent on~$\vert z\vert<p$
and a lacunary series, and hence the unit circle is a natural boundary for~$F$ (see~\cite[Sec.~6.2]{MR2376066}).
It follows that ~$|z|=p^{-1}$ is a natural boundary for~$\zeta_\alpha$.
\end{example}

The overconvergence phenomenon in Example~\ref{typical_full_shift_extension} is visible more generally, and this leads to the following.

\begin{proof}[Proof of Theorem~\ref{single_automorphism_theorem}]
Following the method of~\cite{MR2441142},
there exist function fields~$\mathbb{K}_1,\dots,\mathbb{K}_r$ of the form~$\mathbb{K}_i=\mathbb{F}_{p(i)}(t)$, where each~$p(i)$ is a rational prime, and sets of places~$S_i\subset\mathcal{P}(\mathbb{K}_i)$,~$1\leqslant i\leqslant r$, such that~(\ref{periodic_point_formula_infinite_places}) gives the number of periodic points for~$\alpha$, and the topological entropy is~$\entropy=\sum_{i=1}^r\log p(i)$. Moreover, each set of places~$S_i$ contains the infinite place of~$\mathbb{K}_i$. Subsequently, if~$S_i'\subset\mathcal{P}(\mathbb{K}_i)$ denotes the set~$S_i$ with the infinite place removed,~$1\leqslant i\leqslant r$, then
$\nfix_\alpha(n)=e^{\entropy n}a_n$, where
\begin{equation}\label{i_am_on_a_lonely_road}
a_n=\prod_{i=1}^{r}\prod_{v\in S_i'}|t^n-1|_v.
\end{equation}
Define~$F(z)=\sum_{n\geqslant 1}a_nz^n$ and note that~$F(e^{\entropy}z)=z\zeta_\alpha'(z)/\zeta_\alpha(z)$.

By hypothesis,~$F$ has radius of convergence~$1$. If~$S_i'=\{(t)\}$ for all~$1\leqslant i\leqslant r$, then the product in~(\ref{i_am_on_a_lonely_road}) is trivial, and it follows immediately that~$\zeta_\alpha(z)=(1-e^{\entropy}z)^{-1}$. If not, then there exists~$1\leqslant j\leqslant r$ such that~$S_j'$ contains a place~$w$ corresponding to a polynomial of degree~$d_w\geqslant 1$, with~$|t|_w=1$. Let~$\ell_w$ denote the multiplicative order of the image of~$t$ in the residue field at~$w$ and write~$p=p(i)$. For~$n_k=\ell_w p^k$, a standard argument using the binomial theorem shows that~$|t^{n_k}-1|_w=p^{-d_w p^k}$. For all~$n\geqslant 1$,~$a_n\leqslant |t^n-1|_w$,  so this gives
\[
\limsup_{n_k\rightarrow\infty}a_{n_k}^{1/n_k}\leqslant p^{-d_w/\ell_w},
\]
and~$F$ therefore contains a sequence of partial sums that is uniformly convergent for~$|z|<p^{d_w/\ell_w}$. As in Example~\ref{typical_full_shift_extension}, it follows that the series~$F$ is overconvergent and has the unit circle as a natural boundary. Therefore, the circle~$|z|=e^{-\entropy}$ is a natural boundary for~$\zeta_\alpha$.
\end{proof}

It is clear from the proof above that Theorem~\ref{single_automorphism_theorem} can be improved upon. In particular, overconvergence may be detected even if the function~$F$ appearing in the proof does not have radius of convergence~1 (equivalently,~$\zeta_\alpha$ would have radius of convergence greater than~$e^{-\entropy}$). However, in light of the discussion in the
introduction, we do not pursue this further here.

\section{The zeta function for~$\mathbb{Z}^d$-actions with zero entropy}\label{examples_section}

In this section, we consider~$\mathbb{Z}^d$-actions
of zero entropy
on compact abelian zero-dimensional groups, where $d\geqslant2$.
These are easily classified using dual modules and the entropy formula of Lind, Schmidt and the second author~\cite{MR1062797}. In particular, the action~$\alpha_M$ belongs to this class if and only if for each~$\mathfrak{p}\in\ass(M)$, the ideal
$\mathfrak{p}\mbox{ is non-principal}$ and~$\mathfrak{p}\cap\mathbb{Z}\neq \{0\}$. The most amenable examples in this class are the entropy rank one actions, where each ~$\mathfrak{p}\in\ass(M)$ also satisfies~$\coht(\mathfrak{p})\leqslant 1$. For these actions, we deduce Theorem~\ref{main_result_natural_boundary_entropy_rank_one} from the estimates obtained in Section~\ref{periodic_points_section}. We also include an example which is not of entropy rank one, to emphasize subtler properties of the zeta function
not present in the entropy rank one case. We begin with the well-known example due to Ledrappier~\cite{MR512106}, who used it to demonstrate a~$\mathbb{Z}^2$-action that is mixing but not mixing of higher orders. The dynamical zeta function for this example was first considered by Lind in~\cite{MR1411232}.

\begin{example}[Ledrappier's example]\label{ledrappiers_example}
Let~$\alpha$ denote the~$\mathbb{Z}^2$-action given by the restriction of the full~$\mathbb{Z}^2$-shift to the compact abelian group
\[
\{x=(x_{i,j})\in\mathbb{F}_2^{\mathbb{Z}^2}:
x_{i+1,j}+x_{i,j+1}+x_{i,j}=0\mbox{ for all }(i,j)\in\mathbb{Z}^2\}.
\]
This corresponds to the dual~$R_2$-module~$R_2/\mathfrak{p}$, where~$\mathfrak{p}=(2, 1+u_1+u_2)$, and
\[
R_2/\mathfrak{p}\cong\mathbb{F}_2[t^{\pm 1},(t+1)^{\pm 1}],
\]
for an indeterminate~$t$, where the isomorphism is given by~$u_1\mapsto t$ and~$u_2\mapsto t+1$. Using a geometric argument, Lind shows that there is a constant~$C>0$ such that~$\nfix_\alpha(\Lambda)\leqslant 2^{C\sqrt{[\Lambda]}}$ for all~$\Lambda\in\mathcal{L}_2$. This means that the upper growth rate of periodic points for this is example is~$0$ and using the P{\'o}lya-Carlson Theorem and integrality of coefficients in the Taylor
expansion (see Proposition~\ref{pc_route_to_natural_boundary}), one may deduce that the unit circle is a natural boundary for~$\zeta_\alpha$.

Note that Lind's estimate agrees with that of Theorem~\ref{minkowski_bound}. In particular, following the methods of Section~\ref{periodic_points_section}, we may calculate a value for the constant~$C$. We have~$S(\mathfrak{p})=\{\infty,(t),(t+1)\}$ and, using a standard application of Minkowski's theorem, it can be shown that for any~$\Lambda\in\mathcal{L}_2$, there exists~$\mathbf{n}\in\Lambda$ such that~$||\mathbf{n}||\leq 2\pi^{-1/2}\sqrt{[\Lambda]}$. Hence,
\[
\nfix_\alpha(\Lambda)
\leqslant
\nfix_\alpha(\mathbf{n})
=
\prod_{v\in S(\mathfrak{p})}|t^{n_1}(t+1)^{n_2}-1|_v
\leqslant 2^{|n_1|+|n_2|}
\leqslant 2^{\sqrt{2}||\mathbf{n}||}
\leqslant 2^{2\sqrt{2/\pi}\sqrt{[\Lambda]}}
\]
Lind's calculations also show that this bound cannot be significantly improved upon. In particular, if~$\Lambda=\langle(2^n-1,0), (0,2^n-1)\rangle$,~$n\geqslant 1$,
then~$\nfix_\alpha(\Lambda)=\frac{1}{2}2^{\sqrt{[\Lambda]}}$.

A phenomenon that seems peculiar to this example and minor variations thereof is that~$\nfix_\alpha(\Lambda)$ is trivial for subgroups of~$\mathbb{Z}^d$ of the following form. Suppose that~$[\Lambda]=2^n$,~$n\geqslant 1$. By~(\ref{hermite_matrix}),~$\Lambda=\langle(2^j,0), (b,2^k)\rangle$, where~$j+k=n$ and~$0\leqslant b\leqslant 2^j-1$. Since
\[
\nfix_\alpha((2^j,0))=
|t^{2^j}-1|_\infty
|t^{2^j}-1|_{(t)}
|t^{2^j}-1|_{(t+1)}=2^{2^j}\cdot 1 \cdot 2^{-2^j}=1,
\]
it follows that,~$\nfix_\alpha(\Lambda)=1$.
Note that Theorem~\ref{main_result_restricted_on_primes} identifies a more homogeneous collection of subgroups on which~$\nfix_\alpha(\Lambda)$ is similarly restricted for this and similar examples of entropy rank one.
\end{example}

More generally, under the hypotheses of Theorem~\ref{main_result_natural_boundary_entropy_rank_one},  Corollary~\ref{minkowski_bound} shows that
\[
\ugr(\alpha)=\limsup_{[\Lambda]\rightarrow\infty}
\frac{1}{[\Lambda]}\log\nfix_\alpha(\Lambda)\leqslant
\lim_{[\Lambda]\rightarrow\infty}
\frac{1}{[\Lambda]}\log C^{\sqrt{d}[\Lambda]^{1/d}}
=0.
\]
So,~$\zeta_\alpha$ has radius of convergence~$1$. Subsequently, Theorem~\ref{main_result_natural_boundary_entropy_rank_one} may be deduced as a consequence of Corollary~\ref{minkowski_bound} and the following result.

\begin{proposition}\label{pc_route_to_natural_boundary}
Let~$\alpha$ be a~$\mathbb{Z}^d$-action
by automorphisms of a compact abelian group. If~$\zeta_\alpha$ has radius of convergence 1 and~$d\geqslant 2$, then~$\zeta_\alpha$ admits the unit circle as a natural boundary.
\end{proposition}

\begin{proof}
This is based on Lind's calculations in~\cite[Ex.~3.4]{MR1411232} and is proved using the P{\'o}lya-Carlson Theorem. A crucial step in the proof is to demonstrate that the Taylor series of~$\zeta_\alpha$ has integer coefficients. See~\cite[Lem.~3.3]{MR3336617} and~\cite{MR1411232} for details.
\end{proof}

In the setting above, a natural boundary for~$\zeta_\alpha$ arises via dense singularities on the circle of convergence~$|z|=1$. In the following example,~$\ugr(\alpha)>\mathsf{h}(\alpha)=0$, so the radius of convergence is~$e^{-\ugr(\alpha)}<1$, yet~$\zeta_\alpha$ can be analytically continued beyond~$|z|=e^{-\ugr(\alpha)}$, ultimately encountering a natural boundary at the unit circle. Lind gives an example of an algebraic~$\mathbb{Z}^2$-action of this type~\cite[Ex.~3.3]{MR1411232} which is non-mixing and of entropy rank one. Since  Theorem~\ref{main_result_natural_boundary_entropy_rank_one} precludes the possibility of such examples for mixing entropy rank one~$\mathbb{Z}^d$-actions with~$d\geqslant 2$, to potentially witness this phenomenon for mixing actions, one must either consider~$\mathbb{Z}^2$-actions with positive entropy or consider actions of~$\mathbb{Z}^3$. In light of Theorem~\ref{main_result_natural_boundary_z2}, this leads us to consider the latter case.

\begin{example}[A shift extension of Ledrappier's example]\label{shift extension_of_ledrappier}
By adding an extra dimension to Ledrappier's example, we obtain the compact abelian group
\[
\{x=(x_{i,j,k})\in\mathbb{F}_2^{\mathbb{Z}^3}:
x_{i,j,k}+x_{i+1,j,k}+x_{i,j+1,k}=0\mbox{ for all }(i,j,k)\in\mathbb{Z}^3\},
\]
and we consider the restriction of the full~$\mathbb{Z}^3$-shift to this group. Denote this action by~$\alpha$ and note that the corresponding dual~$R_3$-module is~$D=R_3/\mathfrak{p}$, where~$\mathfrak{p}=(2,1+u_1+u_2)$. Let~$\beta$ denote the~$\mathbb{Z}^2$-action given by the~$R_2$-module~$\mathbb{F}_2[\overline{u}_1^{\pm 1}, \overline{u}_1^{\pm 1}]\subset D$, which corresponds to Ledrappier's example.
If~$\Lambda\in\mathcal{L}_3$ is given by a~$3\times3$ matrix of the form~(\ref{hermite_matrix}), then a straightforward algebraic calculation shows that
\begin{equation}\label{i_dreamed_of_incredible_heights}
\nfix_\alpha(\Lambda)=\nfix_\beta(\Lambda')^{a_3}
\end{equation}
where~$\Lambda'\in\mathcal{L}_2$ is given by the top left~$2\times2$ matrix in~(\ref{hermite_matrix}), with the parameters~$a_1, b_{12}, a_2$ inherited from~$\Lambda$. Note that~$[\mathbb{Z}^3:\Lambda]=[\mathbb{Z}^2:\Lambda']a_3$. Hence,
\begin{eqnarray*}
\ugr(\alpha)
&=&
\limsup_{\Lambda\in\mathcal{L}_3:[\Lambda]\rightarrow\infty}
\frac{1}{[\Lambda]}\log\nfix_\alpha(\Lambda)\\
&=&
\limsup_{\Lambda\in\mathcal{L}_2, n\geqslant 1:[\Lambda]n\rightarrow\infty}
\frac{1}{[\Lambda] n}\log\nfix_\beta(\Lambda)^n\\
&=&
\sup_{\Lambda\in\mathcal{L}_2}
\log\nfix_\beta(\Lambda)^{1/[\Lambda]}.
\end{eqnarray*}

For ease of notation, write~$\Lambda\in\mathcal{L}_2$ in the form
$\Lambda =\langle (a,0), (b,c)\rangle$, where~$a,c\geqslant 1$ satisfy~$ac=[\Lambda]$ and~$0\leqslant b\leqslant a-1$. Let~$T(m)$ denote the set of places corresponding to irreducible polynomials of~$\mathbb{F}_2[t]$ of degree at most~$m$, excluding~$t$ and~$t+1$. By Lemma~\ref{I_caught_sleep},
\begin{equation}\label{the_roses_and_their_prayers}
\nfix_\beta(\Lambda) = \prod_{v\in T(\max\{a,b+c\})} \min\{|t^a-1|^{-1}_v,|t^b(t+1)^c-1|^{-1}_v\}
\end{equation}
Using~(\ref{the_roses_and_their_prayers}), it can easily be verified that~$\nfix_\alpha(\Lambda)\in\{1,4\}$ for all~$\Lambda$ with~$[\Lambda]\leq 6$. For example, when~$\Lambda=\langle (3,0), (1,1)\rangle$,~$\nfix_\beta(\Lambda)=4$, as
\[
|t^3-1|^{-1}_{t^2+t+1}=|t(t+1)-1|^{-1}_{t^2+t+1}=4,
\]
and this is the only place that gives a non-trivial contribution in~(\ref{the_roses_and_their_prayers}). Furthermore, this choice of~$\Lambda$ also gives largest value of~$\nfix_\alpha(\Lambda)^{1/[\Lambda]}=2^{2/3}$ for~$[\Lambda]\leq 6$. When~$[\Lambda]\geqslant 7$,
\[
\nfix_\beta(\Lambda)^{1/[\Lambda]}
\leqslant
2^{2\sqrt{2/(7\pi)}}<2^{2/3}
\]
by the estimate given in Example~\ref{ledrappiers_example}. Thus,~$\ugr(\alpha)=\frac{2}{3}\log 2$ and~$\zeta_\alpha$ has radius of convergence~$2^{-2/3}$.

Given any~$n\geqslant 1$, under the map~$\Lambda\mapsto\Lambda'$, there are precisely~$n^2$ subgroups~$\Lambda\in\mathcal{L}_3$ such that~$[\mathbb{Z}^3:\Lambda]=[\mathbb{Z}^2:\Lambda']n$. Hence, using~(\ref{i_dreamed_of_incredible_heights}),
\[
\zeta_\alpha(z)
=
\exp\left(
\sum_{\Lambda\in\mathcal{L}_2, n\geq 1}
\frac{n\nfix_\beta(\Lambda)^{n}}{[\Lambda]}z^{[\Lambda]n}
\right).
\]
Writing~$\mathcal{J}(N)=\{\Lambda\in\mathcal{L}_2:[\Lambda]\leqslant N-1\}$ and~$\mathcal{K}(N)=\{\Lambda\in\mathcal{L}_2:[\Lambda]\geqslant N\}$, we may split the series inside the brackets into the sum of two series~$F_{\mathcal{J}(N)}(z)$ and~$F_{\mathcal{K}(N)}(z)$ corresponding to this partition of~$\mathcal{L}_2$. Since~$\mathcal{J}(N)$ is finite, and since
\begin{eqnarray*}
F_{\mathcal{J}(N)}(z)
&=&
\sum_{\Lambda\in\mathcal{J}(N), n\geq 1}
\frac{n\nfix_\beta(\Lambda)^{n}}{[\Lambda]}z^{[\Lambda]n}\\
&=&
\sum_{\Lambda\in\mathcal{J}(N)}
\frac{1}{[\Lambda]}
\sum_{n\geq 1}
n(\nfix_\beta(\Lambda)z^{[\Lambda]})^n\\
&=&
\sum_{\Lambda\in\mathcal{J}(N)}
\frac{\nfix_\beta(\Lambda)z^{[\Lambda]}}
{[\Lambda](1-\nfix_\beta(\Lambda)z^{[\Lambda]})^2},
\end{eqnarray*}
we see that~$F_{\mathcal{J}(N)}(z)$ has an analytic continuation to a rational function with~$[\Lambda]$ evenly distributed poles on the circle~$|z|=\nfix_\beta(\Lambda)^{-1/[\Lambda]}$,~$\Lambda\in\mathcal{J}(N)$. On the other hand,
\[
\limsup_{[\Lambda]\geqslant N, n\geq 1:[\Lambda]n\rightarrow\infty}
\left(
\frac{n\nfix_\beta(\Lambda)^{n}}{[\Lambda]}
\right)^{\frac{1}{[\Lambda]n}}
\leqslant
2^{2\sqrt{2/(\pi N)}},
\]
using the estimate obtained for~$\nfix_{\beta}(\Lambda)$. So, the series~$F_{\mathcal{K}(N)}(z)$ has radius of convergence at least~$2^{-2\sqrt{2/(\pi N)}}$. Since~$N$ may be chosen to be arbitrarily large,~$\zeta_\alpha$ has an analytic continuation to the interior of the unit disk. Upon defining~$\Lambda(n)=\langle(2^{n}-1,0),(0,2^n-1)\rangle$, by Example~\ref{ledrappiers_example}, we have~$\nfix_\beta(\Lambda(n))=\frac{1}{2}2^{2^n-1}$, as~$[\Lambda(n)]=(2^n-1)^2$. This means that if~$N>[\Lambda(n)]$, then~$F_{\mathcal{J}(N)}(z)$ has at least~$[\Lambda(n)]$ poles evenly distributed on the circle~$|z|=\nfix_\beta(\Lambda(n))^{-1/[\Lambda(n)]}$. Thus, the poles of~$F_{\mathcal{J}(N)}(z)$ cluster towards a dense set on the unit circle as~$N\rightarrow\infty$.
\end{example}

\section{The zeta function for~$\mathbb{Z}^2$-actions with positive entropy}\label{positive_entropy_section}

In this section we consider a positive entropy~$\mathbb{Z}^2$-action~$\alpha$ by automorphisms of a compact abelian zero-dimensional group~$X$. Let~$\mathbb{P}(\alpha)$ be given by~(\ref{rational_primes_for_the_action}) and let~$M$ denote the dual module. For each~$p\in\mathbb{P}(\alpha)$, the local ring~$M_{(p)}$ may be considered naturally as vector space over the field of fractions of~$R_2/(p)$, which we denote by~$\mathbb{K}(p)$, and~$m(p)=\dim_{\mathbb{K}(p)}M_{(p)}$ is equal to the multiplicity of~$(p)$ in any prime filtration~(\ref{module_chain}) of~$M$. Furthermore, the entropy formula for~$\mathbb{Z}^d$-actions shows that
\begin{equation}\label{entropy_formula}
\entropy(\alpha)=\log
\prod_{p\in\mathbb{P}(\alpha)}p^{m(p)}.
\end{equation}

\begin{theorem}\label{z2_pp_count_theorem}
Let~$\alpha$ be a mixing Noetherian~$\mathbb{Z}^2$-action by automorphisms of a compact abelian zero-dimensional group~$X$. Then there is a mixing Noetherian entropy rank one~$\mathbb{Z}^2$-action~$\beta$ by automorphisms of a compact abelian zero-dimensional group and constants~$A, B>0$ such that for any~$\Lambda\in\mathcal{L}_2$
\[
Ae^{\entropy[L]}
\leqslant
\nfix_\alpha(\Lambda)
\leqslant
Be^{\entropy[L]}
\nfix_\beta(\Lambda),
\]
where~$\entropy=\entropy(\alpha)$ is given by~(\ref{entropy_formula}).
\end{theorem}

The proof proceeds by an induction argument and a series of lemmas that develop the product estimate given by Lemma~\ref{lsw_basic_pp_estimate}. A related problem is dealt with in~\cite[Sec.~7]{MR1062797}, but we require more refined estimates than the ones obtained there.

Recall that if~$\Lambda\in\mathcal{L}_2$, then~(\ref{hermite_matrix}) shows that
the ideal~$\mathfrak{b}_\Lambda$ defined by~(\ref{pp_ideal_definition}) is of the form
\begin{equation}\label{pp_ideal_z2_form}
\mathfrak{b}_\Lambda=(u_1^a-1,u_1^b u^c-1)
\end{equation}
where~$a, c\geqslant 1$,~$0\leqslant b\leqslant a-1$ and~$[\Lambda]=ac$. It is also helpful to note that if~$M$ is a Noetherian~$R_2$-module such that each~$\mathfrak{p}\in\ass(M)$ contains a rational prime, then~\cite[Th.~6.5(3)]{MR1345152} shows that~$M/\mathfrak{b}_\Lambda M$ is finite.

\begin{lemma}\label{pp_count_induction_step_lemma}
Let~$L\subset M$ be Noetherian~$R_2$-modules such that~$M/L\cong R_2/\mathfrak{p}$ for a prime ideal~$\mathfrak{p}\subset R_2$, and assume that~$M/\mathfrak{b}_\Lambda M$ is finite. Then there is an ideal~$\mathfrak{a}\subset R_2$ satisfying ~$\mathfrak{b}_\Lambda\mathfrak{p}\subset\mathfrak{a}\subset\mathfrak{b}_\Lambda\cap\mathfrak{p}$ such that
\[
|M/\mathfrak{b}_\Lambda M|=\left|\frac{R_2/\mathfrak{p}}{\mathfrak{b}_\Lambda(R_2/\mathfrak{p})}\right|\frac{|L/\mathfrak{b}_\Lambda L|}{|(\mathfrak{b}_\Lambda\cap\mathfrak{p})/\mathfrak{a}|}
\]
\end{lemma}

\begin{proof}
Firstly, using~(\ref{pp_ideal_z2_form}), we may assume~$\mathfrak{b}=\mathfrak{b}_\Lambda$ is of the form~$\mathfrak{b}=(b_1, b_2)$.
As in Lemma~\ref{lsw_basic_pp_estimate},
\[
|M/\mathfrak{b}M|
=
 \left|\frac{M/L}{\mathfrak{b}(M/L)}\right|\left|\frac{L/\mathfrak{b}L}{(\mathfrak{b}M\cap L)/\mathfrak{b}L}\right|\\
\]
Hence, we need to show~$(\mathfrak{b}M\cap L)/\mathfrak{b}L\cong (\mathfrak{b}\cap\mathfrak{p})/\mathfrak{a}$, where the ideal~$\mathfrak{a}\subset R_2$ has the desired property. The isomorphism~$M/L\cong R_2/\mathfrak{p}$ implies that there exists~$x\in M$ such that~$M=L+R_2x$, and~$fx\in L$ if and only if~$f\in\mathfrak{p}$. Notice that~$(\mathfrak{b}\cap\mathfrak{p})x + \mathfrak{b}L\subset\mathfrak{b}M\cap L$. On the other hand, if~$y\in \mathfrak{b}M\cap L$, then there exist~$f_1,f_2\in R_2$ and~$l_1,l_2\in L$ such that
\[
y=b_1(f_1 x + l_1)+b_2(f_2 x + l_2)\Rightarrow (b_1 f_1+b_2 f_2)x\in L.
\]
So,~$b_1 f_1 + b_2 f_2\in \mathfrak{p}$, and hence~$y\in (\mathfrak{b}\cap\mathfrak{p})x + \mathfrak{b}L$. Therefore,
\[
\mathfrak{b}M\cap L=(\mathfrak{b}\cap\mathfrak{p})x + \mathfrak{b}L,
\]
and
\[
\frac{\mathfrak{b}M\cap L}{\mathfrak{b}L}
=F,
\frac{(\mathfrak{b}\cap\mathfrak{p})x + \mathfrak{b}L}{\mathfrak{b}L}
\cong
\frac{(\mathfrak{b}\cap\mathfrak{p})x}{\mathfrak{b}L\cap (\mathfrak{b}\cap\mathfrak{p})x}
\]
Finally, since~$\mathfrak{b}\mathfrak{p}x\subset\mathfrak{b}L$, it follows that the natural surjection
\[
\mathfrak{b}\cap\mathfrak{p}
\rightarrow
\frac{(\mathfrak{b}\cap\mathfrak{p})x}{\mathfrak{b}L\cap (\mathfrak{b}\cap\mathfrak{p})x}
\]
has a kernel~$\mathfrak{a}$ containing~$\mathfrak{b}\mathfrak{p}$.
\end{proof}

\begin{lemma}\label{error_factor_control_lemma}
Let~$\mathfrak{p}\subset R_2$ be a non-zero prime ideal and set~$D=R_2/\mathfrak{p}$. Suppose~$D/\mathfrak{b}_\Lambda D$ is finite and~$\mathfrak{a}\subset R_2$ is an ideal satisfying~$\mathfrak{b}_\Lambda \mathfrak{p}\subset\mathfrak{a}\subset\mathfrak{b}_\Lambda \cap\mathfrak{p}$. Then ~$\ass((\mathfrak{b}_\Lambda \cap\mathfrak{p})/\mathfrak{a})\subset\ass(D/\mathfrak{b}_\Lambda D)$ and
\begin{enumerate}
\item\label{coheight_zero_error_factor} If~$\coht(\mathfrak{p})=0$, then~$|(\mathfrak{b}_\Lambda \cap\mathfrak{p})/\mathfrak{a}|=|R_2/\mathfrak{p}|^k$, where~$k\in\{0,1,2\}$, with~$k=0$ if~$\mathfrak{b}_\Lambda \not\subset\mathfrak{p}$.
\item \label{coheight_one_error_factor} If~$\coht(\mathfrak{p})=1$, then there exists a constant~$C$ depending only on~$\mathfrak{p}$
such that
\[
|(\mathfrak{b}_\Lambda \cap\mathfrak{p})/\mathfrak{a}|\leqslant C|D/\mathfrak{b}_\Lambda D|.
\]
\item \label{principal_prime_error_factor} If~$\mathfrak{p}=(p)$ for a rational prime~$p$, then~$|(\mathfrak{b}\cap\mathfrak{p})/\mathfrak{a}|=1$.
\end{enumerate}
\end{lemma}

\begin{proof}
Using~(\ref{pp_ideal_z2_form}), we may assume~$\mathfrak{b}=\mathfrak{b}_\Lambda$ is of the form~$\mathfrak{b}=(b_1, b_2)$.
Furthermore, since~$\mathfrak{b}\mathfrak{p}\subset\mathfrak{a}$, all elements of~$(\mathfrak{b}\cap\mathfrak{p})/\mathfrak{a}$ are annihilated by~$\mathfrak{b}+\mathfrak{p}$. Therefore,
\[
\ass((\mathfrak{b}\cap\mathfrak{p})/\mathfrak{a})
\subset
\ass(R_2/(\mathfrak{b}+\mathfrak{p}))
=
\ass((R_2/\mathfrak{p})/\mathfrak{b}(R_2/\mathfrak{p}))
\]

The surjective module homomorphism~$R_2\oplus R_2 \rightarrow \mathfrak{b}/\mathfrak{a}$ given by
\[
f_1\oplus f_2\mapsto f_1 b_1 + f_2 b_2 + \mathfrak{a}
\]
has~$\mathfrak{p}\oplus\mathfrak{p}$ in the kernel as~$\mathfrak{b}\mathfrak{p}\subset\mathfrak{a}$. Hence, this map also naturally defines a surjective module homomorphism~$\phi:D\oplus D\rightarrow \mathfrak{b}/\mathfrak{a}$. Denote the images of~$b_1$ and~$b_2$ in~$D$ by~$\overline{b}_1$ and ~$\overline{b}_2$ respectively. If
\[
K=\{g_1\oplus g_2\in D\oplus D:g_1\overline{b}_1+g_2\overline{b}_2=0\},
\]
then~$\phi(K)$ is also the image of~$\mathfrak{b}\cap\mathfrak{p}$ in~$\mathfrak{b}/\mathfrak{a}$. That is,~$\phi(K)\cong (\mathfrak{b}\cap\mathfrak{p})/\mathfrak{a}$. Furthermore,~$\mathfrak{b}K\subset\ker(\phi)$, as~$\mathfrak{b}\mathfrak{p}\subset\mathfrak{a}$. Therefore,
\begin{equation}\label{if_you_ever}
|(\mathfrak{b}\cap\mathfrak{p})/\mathfrak{a}|\mbox{ divides }|K/\mathfrak{b}K|.
\end{equation}

If~$\coht(\mathfrak{p})=0$, then~$\mathfrak{p}$ is maximal and~$D$ is hence a finite field. Therefore,~$K/\mathfrak{b}K$ is a vector space of dimension at most~$2$ over~$D$, and hence the same is true for the homomorphic image~$(\mathfrak{b}\cap\mathfrak{p})/\mathfrak{a}$. Furthermore, since all elements of~$(\mathfrak{b}\cap\mathfrak{p})/\mathfrak{a}$ are annihilated by~$\mathfrak{b}+\mathfrak{p}$, if~$\mathfrak{b}\not\subset\mathfrak{p}$,~$\mathfrak{b}+\mathfrak{p}=R_2$ and so~$(\mathfrak{b}\cap\mathfrak{p})/\mathfrak{a}$ must be trivial as it is annihilated by~$R_2$. This gives~(\ref{coheight_zero_error_factor}).

Now assume~$\coht(\mathfrak{p})=1$. If~$\mathfrak{c}\subset D$ is a non-trivial ideal, then Lemma~\ref{integral_closure_trick} shows that there exist  constants~$A, B>0$ independent of~$\mathfrak{c}$ such that
\[
A|E/\mathfrak{c}E|\leqslant|D/\mathfrak{c}|\leqslant B|E/\mathfrak{c}E|
\]
where~$E$ is the integral closure of~$D$ in its field of fractions. Furthermore, we may exploit the favourable ideal arithmetic in~$E$, as~$E$ is a Dedekind domain.
By assumption, the image~$\overline{\mathfrak{b}}$ of~$\mathfrak{b}$ in~$D$ is non-trivial. This means that at most one of~$\overline{b}_1$ and~$\overline{b}_2$ is zero. Furthermore, if precisely one of~$\overline{b}_1$ or~$\overline{b}_2$ is zero, then~$K\cong D$ and~(\ref{coheight_one_error_factor}) follows immediately from~(\ref{if_you_ever}). Hence, assume~$\overline{b}_1\neq0$ and~$\overline{b}_2\neq 0$. In this case,~$K\cong\mathfrak{c}=\{g\in D:g\overline{b}_1\in(\overline{b}_2)\}$ which is a non-trivial ideal of~$D$. Hence,
\begin{equation}\label{change_your_mind}
\left|\frac{K}{\mathfrak{b}K}\right|
=\left|\frac{\mathfrak{c}}{\overline{\mathfrak{b}}\mathfrak{c}}\right|
=\frac{|D/\overline{\mathfrak{b}}\mathfrak{c}|}{|D/\mathfrak{c}|}
\leqslant
\frac{B|E/\overline{\mathfrak{b}}\mathfrak{c}E|}{|D/\mathfrak{c}|}
=
\frac{B|E/\overline{\mathfrak{b}}E||E/\mathfrak{c}E|}{|D/\mathfrak{c}|},
\end{equation}
where the final equality holds since~$E$ is a Dedekind domain. Since
\[
|E/\overline{\mathfrak{b}}E||E/\mathfrak{c}E|
\leqslant
A^{-2}|D/\overline{\mathfrak{b}}||D/\mathfrak{c}|,
\]
(\ref{if_you_ever}) and (\ref{change_your_mind}) give
\[
|(\mathfrak{b}\cap\mathfrak{p})/\mathfrak{a}|\leqslant|K/\mathfrak{b}K|\leqslant BAF,^{-2}|D/\overline{\mathfrak{b}}|,
\]
thus yielding~(\ref{coheight_one_error_factor}) with~$C=BA^{-2}$.

Finally, if~$\mathfrak{p}=(p)$ for a rational prime~$p$, then it is straightforward to show that~$\mathfrak{b}\cap\mathfrak{p}=\mathfrak{b}\mathfrak{p}$, so~(\ref{principal_prime_error_factor}) follows immediately (note that this is also shown more generally for principal ideals with a trivial complex unitary variety in~\cite[Lem.~7.6]{MR1062797}).
\end{proof}

\begin{remark}\label{a_red_red_rose}
From Lemma~\ref{error_factor_control_lemma}, it follows that if~$L\subset M$ are Noetherian~$R_2$-modules such that each~$\mathfrak{p}\in\ass(M)$ contains a rational prime, then there is a constant~$C\geq 0$ such that for all~$\Lambda\in\mathcal{L}_2$
\[
|M/\mathfrak{b}_\Lambda M|>C|L/\mathfrak{b}_\Lambda L|.
\]
This may be proved by induction as follows.
Firstly, note that by lifting a prime filtration of~$M/L$ to~$M$, we obtain a sequence of modules of the form~(\ref{module_chain}) with~$M_0=L$. Then,
$M_i/M_{i-1}\cong R_2/\mathfrak{p}_i$, where~$\mathfrak{p}_i$ is a prime ideal satisfying one of the three cases in Lemma~\ref{error_factor_control_lemma}. Furthermore, by Lemma~\ref{pp_count_induction_step_lemma}
\[
|M_i/\mathfrak{b}_\Lambda M_i|=
\left|\frac{R_2/\mathfrak{p}_i}{\mathfrak{b}_\Lambda(R_2/\mathfrak{p}_i)}\right|\frac{|M_{i-1}/\mathfrak{b}_\Lambda M_{i-1}|}
{|(\mathfrak{b}_\Lambda\cap\mathfrak{p}_i)/\mathfrak{a}_i|},
\]
where~$\mathfrak{a}_i\subset R_2$ is an ideal satisfying~$\mathfrak{b}_\Lambda \mathfrak{p}_i\subset\mathfrak{a}_i\subset\mathfrak{b}_\Lambda \cap\mathfrak{p}_i$. From Lemma~\ref{error_factor_control_lemma}, it follows that there is a constant~$C_i>0$, independent of~$\Lambda$, such that
\[
|M_i/\mathfrak{b}_\Lambda M_i|\geqslant
C_i|M_{i-1}/\mathfrak{b}_\Lambda M_{i-1}|
\]
Therefore,~$|M/\mathfrak{b}_\Lambda M|\geqslant \left(\prod_{i=1}^r C_i\right)|L/\mathfrak{b}_\Lambda L|$.
\end{remark}

\begin{lemma}\label{non_mixing_prime_lemma}
Suppose that~$L\subset M$ are Noetherian~$R_2$-modules such that each~$\mathfrak{p}\in\ass(M)$ contains a rational prime, and suppose that~$M/L\cong R_2/\mathfrak{q}$ for a prime ideal~$\mathfrak{q}\subset R_2$ which contains a polynomial of the form~$u^{\mathbf{n}}-1$,
~$\mathbf{n}\in\mathbb{Z}^d\setminus\{0\}$. If~$\alpha_M$ is mixing, then there exist constants~$A, B >0$ such that for all~$\Lambda\in\mathcal{L}_2$
\[
A|L/\mathfrak{b}_\Lambda L|
\leqslant
|M/\mathfrak{b}_\Lambda M|
\leqslant
B|L/\mathfrak{b}_\Lambda L|.
\]
\end{lemma}

\begin{proof}
First note that the lower estimate is provided immediately by Remark~\ref{a_red_red_rose}. For the upper estimate, let~$f=u^{\mathbf{n}}-1$ and consider the~$R_2$-module~$fM$. Since~$\alpha_M$ is mixing, multiplication by~$f$ on~$M$ is injective, so~$M$ and~$fM$ are isomorphic as~$R_2$-modules. On the other hand~$f\in\mathfrak{p}$, so~$fM\subset L$ and Remark~\ref{a_red_red_rose} again provides a constant~$C>0$ such that
\[
|L/\mathfrak{b}L|\geqslant
C|fM/\mathfrak{b}_\Lambda fM|=C|M/\mathfrak{b}_\Lambda M|.
\]
Dividing by~$C$ gives the required upper estimate.
\end{proof}

\begin{proof}[Proof of Th.~\ref{z2_pp_count_theorem}]
We begin by taking a prime filtration of the dual module~$M$ of the form~(\ref{module_chain}). Each prime~$\mathfrak{p}_i$ appearing in such a filtration  contains a rational prime, since~$X$ is zero-dimensional. If~$u^{\mathbf{n}}-1\not\in\mathfrak{p}_i$ for all~$\mathbf{n}\in\mathbb{Z}^d\setminus\{0\}$, we say that~$\mathfrak{p}_i$ is \emph{mixing}.
Note that whenever~$\mathfrak{p}_i\in\ass(M)$,~$\mathfrak{p}_i$ is mixing, as~$\alpha$~is mixing.
When~$\coht(\mathfrak{p}_i)=0$,~$\mathfrak{p}_i$ is maximal and~$R_2/\mathfrak{p}_i$ is a finite field, so~$\mathfrak{p}_i$ is not mixing and~$\mathfrak{p}_i\not\in\ass(M)$. Hence, each prime~$\mathfrak{p}_i$ is of one of the following three types.
\begin{enumerate}
\item \label{prime_type_principal}~$\mathfrak{p}_i=(p_i)$ for a rational prime~$p_i$ and~$\mathfrak{p}_i\in\ass(M)$,
\item \label{prime_type_non_principal_mixing}~$\coht(\mathfrak{p}_i)=1$ and~$\mathfrak{p}_i$ is mixing,
\item \label{prime_type_non_mixing}~$\coht(\mathfrak{p}_i)\leqslant1$ and~$\mathfrak{p}_i$ is not mixing.
\end{enumerate}

If~$\mathfrak{p}_i$ is of type (\ref{prime_type_principal}), a straightforward calculation shows that for any~$\Lambda\in\mathcal{L}_2$,
\[
\left|\frac{R_2/\mathfrak{p}_i}{\mathfrak{b}_\Lambda(R_2/\mathfrak{p}_i)}\right|
=
p_i^{[\Lambda]}
\]
and hence Lemma~\ref{pp_count_induction_step_lemma} and Lemma~\ref{error_factor_control_lemma}(\ref{principal_prime_error_factor}) give
\begin{equation}\label{type_1_induction_step}
|M_i/\mathfrak{b}_\Lambda M_i|
=p_i^{[\Lambda]}|M_{i-1}/\mathfrak{b}_\Lambda M_{i-1}|
\end{equation}

If~$\mathfrak{p}_i$ is of type (\ref{prime_type_non_principal_mixing}), then Lemma~\ref{pp_count_induction_step_lemma} and Lemma~\ref{error_factor_control_lemma}(\ref{coheight_one_error_factor}) show that there are constants~$A_i, B_i>0$ such that for all~$\Lambda\in\mathcal{L}_2$
\begin{equation}\label{type_2_induction_step}
A_i \leqslant |M_i/\mathfrak{b}_\Lambda M_i|
\leqslant
B_i
\left|\frac{R_2/\mathfrak{p}_i}{\mathfrak{b}_\Lambda(R_2/\mathfrak{p}_i)}\right|
|M_{i-1}/\mathfrak{b}_\Lambda M_{i-1}|.
\end{equation}

Finally, if~$\mathfrak{p}_i$ is of type (\ref{prime_type_non_mixing}), then Lemma~\ref{pp_count_induction_step_lemma} and Lemma~\ref{non_mixing_prime_lemma} show that
there are positive constants~$A_i, B_i>0$ such that for all~$\Lambda\in\mathcal{L}_2$
\begin{equation}\label{type_3_induction_step}
A_i |M_{i-1}/\mathfrak{b}_\Lambda M_{i-1}|
\leqslant
|M_i/\mathfrak{b}_\Lambda M_i|
\leqslant
B_i
|M_{i-1}/\mathfrak{b}_\Lambda M_{i-1}|.
\end{equation}

Recall that any prime of type~(\ref{prime_type_principal}) appears in the filtration~(\ref{module_chain}) with a fixed multiplicity, as given in the entropy formula~(\ref{entropy_formula}). Hence,  applying~(\ref{type_1_induction_step}),~(\ref{type_2_induction_step}) and~(\ref{type_3_induction_step}) inductively, the required result follows, where~$\beta=\beta_L$ is simply defined by an~$R_2$-module~$L$ which is a direct sum of the modules~$R_2/\mathfrak{p}_i$, where~$\mathfrak{p}_i$ runs through the primes of type~(\ref{prime_type_non_principal_mixing}).
\end{proof}

When combined with Corollary~\ref{minkowski_bound}, Theorem~\ref{z2_pp_count_theorem} has the following consequence which we shall apply to demonstrate the existence of a natural boundary for the dynamical zeta function. For the rest of this section, denote the multiplicative group~$\langle p:p\in\mathbb{P}(\alpha)\rangle\subset\mathbb{Q}^\times$ by~$G(\alpha)$ and note that all values~$\nfix_\alpha(\Lambda)$,~$\Lambda\in\mathcal{L}_2$, lie in~$G(\alpha)\cap\mathbb{N}$.

\begin{corollary}\label{pp_count_final_product}
Let~$\alpha$ be a mixing Noetherian~$\mathbb{Z}^2$-action by automorphisms of a compact abelian zero-dimensional group~$X$. Then there exist constants~$A\in G(\alpha)$,~$C>0$, and a function~$\psi:\mathcal{L}_2\rightarrow G(\alpha)\cap\mathbb{N}$ such that for all~$\Lambda\in\mathcal{L}_2$, ~$\psi(\Lambda)\leqslant C^{\sqrt{[\Lambda]}}$ and
\begin{equation}\label{pp_count_final_product_terms} \nfix_\alpha(\Lambda)=Ae^{\entropy[\Lambda]}\psi(\Lambda).
\end{equation}
\end{corollary}

\begin{proof}
Standard properties of torsion abelian groups show that we may write~$X=\oplus_{p\in\mathbb{P}(\alpha)}X_p$, where each~$X_p$ is a~$p$-group and is also~$\alpha$-invariant. Write~$\alpha_p$ for the restriction of~$\alpha$ to~$X_p$, so that for all~$\Lambda\in\mathcal{L}$
\begin{equation}\label{p_group_pp_product}
\nfix_\alpha(\Lambda)=\prod_{p\in\mathbb{P}(\alpha)}\nfix_{\alpha_p}(\Lambda).
\end{equation}
Theorem~\ref{z2_pp_count_theorem} shows that given any~$p\in\mathbb{P}(\alpha)$, there exists a constant~$A_p>0$ such that for all~$\Lambda\in\mathcal{L}_2$,~$\nfix_{\alpha_p}(\Lambda) \geqslant A_p\,p^{m(p)[L]}$,
where~$m(p)$ is as in~(\ref{entropy_formula}). Since~$\nfix_{\alpha_p}(\Lambda)$ is a power of~$p$, we can therefore assume~$A_p$ is a power of~$p$ and that there is a function~$\psi_p:\mathcal{L}_2\rightarrow \{p^n:n\geqslant 0\}$ such that
\[
\nfix_{\alpha_p}(\Lambda) = A_p\,p^{m(p)[L]}\psi_p(\Lambda).
\]
So, (\ref{pp_count_final_product_terms}) follows from~(\ref{p_group_pp_product}) and~(\ref{entropy_formula}) with~$A=\prod_{p\in\mathbb{P}(\alpha)}A_p$ and~$\psi=\prod_{p\in\mathbb{P}(\alpha)}\psi_p$. Furthermore, for each~$p\in\mathbb{P}(\alpha)$, the upper estimates provided by Theorem~\ref{z2_pp_count_theorem} and Corollary~\ref{minkowski_bound} show that there is a constant~$C_p>0$ such that for all~$\Lambda\in\mathcal{L}_2$
\[
\psi_p(\Lambda)\leqslant C_p^{\sqrt{[\Lambda]}}.
\]
The required upper estimate for~$\psi$ now follows with~$C=\prod_{p\in\mathbb{P}(\alpha)}C_p$.
\end{proof}

We now use Corollary~\ref{pp_count_final_product} and Theorem~\ref{main_result_restricted_on_primes} to prove Theorem~\ref{main_result_natural_boundary_z2}.
In order to convert an
implausibility into an impossibility, we will also need some elementary bounds
for the sum of divisors function~$\sigma$
along the lines of Gronwall's theorem~\cite{MR1500940}.
In addition to showing that
\[
\limsup_{n\to\infty}\frac{\sigma(n)}{n\log\log n}=e^{\gamma},
\]
Gronwall showed that the sequence~$(n_k)$
defined by
\[
n_k=(p_1p_2\cdots p_k)^{\lfloor\log p_n\rfloor},
\]
where~$p_1,p_2,\dots$ are the rational primes in their
natural order, has
\[
\lim_{k\to\infty}\frac{\sigma(n_k)}{n_k\log\log n_k}=e^{\gamma}.
\]
We need lower bounds infinitely often
along arithmetic progressions
for the same expression,
but require very little information about exact
values. We thank Stephan Baier for suggesting
how to do this, and Pieter Moree for further
suggestions on how the methods of Choie {\it{et al.}}~\cite[Sec.~4]{MR2394891}
giving explicit constants for the upper limit with~$n$ restricted to being
odd, square-free, square-full, and so on
should give a complete picture.

\begin{lemma}\label{lyingunderneaththestars}
Given integers~$t,q\ge1$ there is a constant~$C>0$
with
\[
\frac{\sigma(n)}{n\log\log n}\ge C
\]
for infinitely
many values of~$n$ congruent to~$t$ modulo~$q$.
\end{lemma}

\begin{proof}
Write~$p$ for a prime running through the set of primes
in their natural order.
For $k\ge1$ let
\[
n_k=\prod_{p\le k\atop p{\mathrel{\kern-2pt\not\!\kern3.5pt|}}q}p
\]
be the product of
all primes not dividing~$q$ and less than or equal to~$k$.
By the Chinese remainder theorem there is some~$N_k\le qn_k$
with~$N_k\equiv t$ modulo~${q}$ and~$N_k\equiv 0$ modulo~${n_k}$.
By construction
\[
\sigma(n_k)\ge N_k\prod_{p\le k\atop p{\mathrel{\kern-2pt\not\!\kern3.5pt|}}q}
\left(1+\textstyle\frac1p\right)
\sim C(q)\log k
\]
for some constant~$C(q)>0$
by Mertens' theorem. On the other hand,
\[
k\asymp\log n_k\asymp\log N_k,
\]
so
\[
\sigma(N_k)\gg N_k\log\log N_k
\]
for infinitely many~$k$ as required.
\end{proof}

\begin{proof}[Proof of Theorem~\ref{main_result_natural_boundary_z2}] By Corollary~\ref{pp_count_final_product}
\[
\zeta_\alpha(z)=
\exp\left(A\sum_{n\geqslant 1} a_n(e^\entropy z)^n\right)
\]
where~$a_n=\sum_{[L]=n}\psi(L)$.
There are exactly~$\sigma(n)$ subgroups~$\Lambda\in\mathcal{L}_2$ such that~$[\Lambda]=n$. Furthermore,
Corollary~\ref{pp_count_final_product} shows that~$\psi(\Lambda)\leqslant C^{\sqrt{[\Lambda]}}$, so~$a_n$ satisfies
\begin{equation}\label{bounds_for_taylor_coefficients}
\sigma(n)\leqslant a_n\leqslant \sigma(n)\cdot C^{\sqrt{[\Lambda]}},
\end{equation}
whereby~$\limsup_{n\rightarrow\infty}a_n^{1/n}=1$.
Thus the series~$F(z)=\sum_{n\geqslant 1} a_n z^n$ has radius
of convergence~$1$ and~$\zeta_\alpha$ has radius of convergence~$e^{-\entropy}$. Furthermore,
\[
F(e^{\entropy}z)=\frac{z\zeta_\alpha'(z)}{A\zeta_\alpha(z)},
\]
so if~$\zeta_\alpha$ has analytic continuation beyond the circle~$|z|=e^{-\entropy}$ then~$F$ has analytic continuation beyond the unit circle. We assume this to be the case for a contradiction.

Since~$F$ has integer coefficients and radius of convergence~1, the P{\'o}lya-Carlson theorem implies that~$F$ is a rational function. Therefore,~$(a_n)$ is an integer linear recurrence sequence.
Using~\cite[Th.~1.2]{MR1990179}, the sequence~$(a_n)$ may be decomposed into finitely many non-degenerate subsequences corresponding to a finite set of congruence classes,
as no subsequence can be identically zero
by~(\ref{bounds_for_taylor_coefficients}). Each such subsequence is
given by a linear recurrence with
characteristic roots that are algebraic integers,
as it is comprised of positive
integers.
None of these linear recurrence sequences can have a characteristic root of absolute
value
greater than~$1$ by~\cite[Th.~2.3]{MR1990179}
and the upper bound from~\eqref{bounds_for_taylor_coefficients}.
Now by Theorem~\ref{main_result_restricted_on_primes}
there is some integer constant~$A$ with
\[
a_p=A\sigma(p)=A(p+1)
\]
for infinitely many primes~$p$, so
by the pigeon-hole principle one of the
non-degenerate linear recurrence sequences, say~$(b_n=a_{qn+t})$
has integer constants~$A'$ and~$t'$ with infinitely many terms satisfying
\begin{equation}\label{equation:sherlockholmes}
b_n=a_{qn+t}=A'n+t'.
\end{equation}
The characteristic roots of the sequence~$(b_n)$
and their conjugates
are algebraic integers,
since the sequence is
integral, and have modulus no greater than~$1$.
By Kronecker's lemma~\cite[Th.~1.31]{MR1700272}, they must therefore be roots
of unity, so there can only be one since the
sequence is non-degenerate.
Thus, by~\cite[Eq.~(1.2)]{MR1990179}, we may write~$b_n=P(n)\omega^n$
for some polynomial~$P$ and root of unity~$\omega$.
On the other hand,~\eqref{equation:sherlockholmes}
shows that~$\vert P(n)\omega^n\vert\ll n$ for
infinitely many prime values of~$n$, so~$P$ is linear.
Therefore,~$b_n=A'n+t'$
for all~$n$, and so~$a_n=A''n+t''$ for
some constants~$A''$ and~$t''$ for all~$n$
lying in an arithmetic progression.
However, according to the lower bound in~(\ref{bounds_for_taylor_coefficients}) and Lemma~\ref{lyingunderneaththestars} this is impossible, thus giving the required contradiction.
\end{proof}

\bibliographystyle{plain}

\end{document}